\documentclass[11pt]{amsart} % use larger type; default would be 10pt

\usepackage[utf8]{inputenc} % set input encoding (not needed with XeLaTeX)

  \usepackage{fullpage}

\usepackage{graphicx} % support the \includegraphics command and options

\usepackage{amsmath}
\usepackage{amssymb}
\usepackage{color}
\usepackage{amstext}
\usepackage{amsthm}
\usepackage{url}
\usepackage{booktabs} % for much better looking tables
\usepackage{array} % for better arrays (eg matrices) in maths
\usepackage{paralist} % very flexible & customisable lists (eg. enumerate/itemize, etc.)
\usepackage{verbatim} % adds environment for commenting out blocks of text & for better verbatim
\usepackage{subfig} % make it possible to include more than one captioned figure/table in a single float
\usepackage{fancyhdr} % This should be set AFTER setting up the page geometry
\title{The parametric Frobenius problem and parametric exclusion}
\newcommand\blfootnote[1]{%
  \begingroup
  \renewcommand\thefootnote{}\footnote{#1}%
  \addtocounter{footnote}{-1}%
  \endgroup
}

\begin{document}
\maketitle
{\centering Bobby Shen\footnote{Bobby Shen,  Department of Mathematics, Massachusetts Institute of Technology,
Cambridge, Massachusetts, USA, \url{runbobby@mit.edu}} \par}

\newtheorem{theorem}{Theorem}[section]
\newtheorem{corollary}{Corollary}[theorem]
\newtheorem{lemma}[theorem]{Lemma}
\newtheorem{proposition}[theorem]{Proposition}
\newtheorem{conjecture}[theorem]{Conjecture}
 
\theoremstyle{definition}
\newtheorem{definition}[theorem]{Definition}
 
\theoremstyle{remark}
\newtheorem*{remark}{Remark}
\begin{abstract}
The Frobenius number of relatively prime positive integers $a_1, \ldots, a_n$ is the largest integer that is not a nononegative integer combination of the $a_i.$ Given positive integers $a_1, \ldots, a_n$ with $n \ge 2,$ the set of multiples of $\gcd(a_1, \ldots, a_n)$ which have less than $m$ distinct representations as a nonnegative integer combination of the $a_i$ is bounded above, so we define $f_{m, \ell}(a_1, \ldots, a_n)$ to be the $\ell^{\text{th}}$ largest multiple of $\gcd(a_1, \ldots, a_n)$ with less than $m$ distinct representations (which generalizes the Frobenius number) and $g_m(a_1, \ldots, a_n)$ to be the number of positive multiples of $\gcd(a_1, \ldots, a_n)$ with less than $m$ distinct representations. In the parametric Frobenius problem, the arguments are polynomials. Let $P_1, \ldots, P_n$ be integer valued polynomials of one variable which are eventually positive. We prove that $f_{m, \ell}(P_1(t), \ldots, P_n(t))$ and $g_m(P_1(t), \ldots, P_n(t)),$ as functions of $t,$ are eventually quasi-polynomial. A function $h$ is eventually quasi-polynomial if there exist $d$ and polynomials $R_0, \ldots, R_{d-1}$ such that for such that for sufficiently large integers $t,$ $h(t)=R_{t \pmod{d}}(t).$ We do so by formulating a type of parametric problem that generalizes the parametric Frobenius Problem, which we call a parametric exclusion problem. We prove that the $\ell^{\text{th}}$ largest value of some polynomial objective function, with multiplicity, for a parametric exclusion problem and the size of its feasible set are eventually quasi-polynomial functions of $t.$
\end{abstract}

\blfootnote{2010 \textit{Mathematics Subject Classsifications}: 05A16, 11D07, 90C10, 90C31} 
\section{Introduction}

\subsection{The Parametric Frobenius Problem}

Let $a_1, \ldots, a_n$ be positive integers with greatest common divisor $1.$ The \textit{Frobenius number} $F(a_1, \ldots, a_n)$ is the largest integer that is not of the form $\sum_{i=1}^n b_i a_i$ for nonnegative integers $b_1, \ldots,  b_n.$ A related quantity is the number of positive integers which are not of this form, which we denote by $G(a_1, \ldots, a_n).$ Determining $F(a_1, \ldots, a_n).$ is known as the \textit{Frobenius Problem}, which has been well studied. In 1884, Sylvester proved that $F(a,b)=ab-a-b$ \cite{Syl}.  It is known that $G(a,b)=\frac{1}{2}(a-1)(b-1).$ Ram\'irez Alfons\'in's book is a broad survey of known results \cite{Ram}.

We also define these quantities for general positive integers $a_1, \ldots, a_n.$ Define $F(a_1, \ldots, a_n)$ to be the largest multiple of $\gcd(a_1, \ldots, a_n)$ which is not a $\mathbb{Z}_{\ge 0}$ linear combination of $a_1, \ldots, a_n,$ and define $G(a_1, \ldots, a_n)$ to be the number of positive multiples of $\gcd(a_1, \ldots, a_n)$ which are not $\mathbb{Z}_{\ge 0}$ linear combinations. It is easy to check that for positive integers $c, a_1, \ldots, a_n,$ $F(ca_1, \ldots, ca_n)=c F(a_1, \ldots, a_n),$ and $G(ca_1, \ldots, ca_n) = G(a_1, \ldots, a_n).$

The \textit{parametric Frobenius problem}, as defined by Roune and Woods \cite{RW}, is to determine, given functions $P_1, \ldots, P_n : \mathbb{Z} \rightarrow \mathbb{Z},$ $F(P_1(t), \ldots, P_n(t))$ as a function of the integer parameter $t.$ In this paper, $t$ is always a positive integer. For example, using Sylvester's result, for all integers $t > 2,$
\begin{displaymath} F(t, t-2) = \left \{ \begin{array}{ll} t(t-2)-t-(t-2) & t \equiv 1 \pmod{2} \\ 2\left(\frac{t}{2}\frac{t-2}{2} - \frac{t}{2}-\frac{t-2}{2}\right) & t \equiv 0 \pmod{2.} \\ \end{array} \right. \end{displaymath}

The two pieces are not the same polynomial. Rather, $F(t, t+2)$ is eventually quasi-polynomial.

\begin{definition}\label{deqp} $f$ is \textit{eventually quasi-polynomial} (EQP) if its domain is a subset of $\mathbb{Z}$ that contains sufficiently large integers and there exists a positive integer $d$ and polynomials $R_0, \ldots, R_{d-1}$ in $\mathbb{R}[u]$ such that for sufficiently large $t,$ $f(t)=R_{t \pmod{d}}(t).$ 
\end{definition}

We say that $d$ is a period of $f$ if polynomials $R_0, \ldots, R_{d-1}$ exist as above, and we call $R_0, \ldots, R_{d-1}$ the components. Similarly, one can show that $G(P_1(t), \ldots, P_n(t))$ is eventually quasi-polynomial in this example. We define functions which generalize the Frobenius problem using the following proposition.

\begin{proposition} \label{p1} Let $n$ be a positive integer at least $2$ and $a_1, \ldots, a_n$ be relatively prime positive integers. For integers $k,$ define $h(k)$ to be the number of nonnegative integer $n$-tuples $(b_1, \ldots, b_n)$ such that $k = \sum_{i=1}^n b_i a_i.$ Then $\lim_{k \rightarrow \infty} h(k) = \infty.$ \end{proposition}

\begin{proof} Let $N$ be a positive integer. For all integers $k$ greater than $N a_1 a_2 + F(a_1, \ldots, a_n),$ $\\ k-Na_1 a_2 > F(a_1, \ldots, a_n),$ so there exist $c_1, \ldots, c_n$ in $\mathbb{Z}_{\ge 0}$ such that $k-Na_1 a_2 = \sum_{i=1}^n c_i a_i.$ Since $n \ge 2,$ for all integers $q$ with $0 \le q \le N,$
\begin{displaymath} k=\left(\sum_{i=3}^n c_i a_i \right) + (c_1 + q a_2) a_1 + (c_2 + (N-q) a_1)a_2. \end{displaymath}
All coefficients for all values of $\ell$ are nonnegative integers, which gives us $N+1$ distinct representations, which shows that $h(k)$ goes to infinity. \end{proof}

\begin{definition} Let $n, m, \ell, a_1, \ldots, a_n$ be positive integers, where $n \ge 2.$ 

$F_{m, \ell}(a_1, \ldots, a_n)$ is the $\ell^{\text{th}}$ largest integer in the set $\{k \in \mathbb{Z} \gcd(a-1, \ldots, a_n) \mid h(k) < m\},$ where $h(k)$ is as above.

$G_{m}(a_1, \ldots, a_n)$ is the number of positive multiples of $\gcd(a_1, \ldots, a_n)$ $k$ such that $h(k)<m.$ \end{definition}

These numbers exist by above Proposition \ref{p1}. These new functions are more general than the usual functions for the Frobenius problem: \begin{displaymath} F_{1, 1}(a_1, \ldots, a_n) = F(a_1, \ldots, a_n). \quad G_1(a_1, \ldots, a_n) = G(a_1, \ldots, a_n). \end{displaymath} As before, for positive integers $c, a_1, \ldots, a_n,$ we have $F_{m, \ell}(ca_1, \ldots, ca_n)=cF_{m,\ell}(a_1, \ldots, a_n)$ and $G_m(ca_1, \ldots, ca_n)=G_m(a_1, \ldots, a_n).$
 
Again, one can consider the case when the argumentss are polynomials of $t.$ In order for the functions to be defined for sufficiently large $t,$ we should have polynomials $P_1, \ldots, P_n$ that map $\mathbb{Z}$ to $\mathbb{Z}$ and have positive leading coefficient. In 2015, Roune and Woods conjectured the following for such polynomials \cite{RW}.

\begin{conjecture} $P_1, \ldots, P_n$ be in $\mathbb{R}[u]$ which map $\mathbb{Z}$ to $\mathbb{Z}$ and have positive leading coefficient. For all $t$ such that $P_i(t)>0$ for all $i,$ let $D(t)=F(P_1(t), \ldots, P_n(t)).$ Then $D$ is EQP.
\end{conjecture}

Roune and Woods conjectured that for such polynomials, $F(P_1(t), \ldots, P_n(t))$ is EQP. They actually conjectured this result when $P_1, \ldots, P_n$ are EQP, map $\mathbb{Z}$ to $\mathbb{Z}$ and are eventually positive. Since $n$ is finite, it's easy to see that this broader case reduces to the case of polynomials. They proved the conjecture for the case $n\le 3$ and the case when all polynomials have degree at most $1$ and proved certain results about the period of the resulting EQP and its degree (or the maximum degree of the components). In this paper, we do not address the degree or period of the EQPs. Otherwise, we prove a far more general result about the parametric Frobenius problem.

\begin{theorem} \label{main} Let $n, m,$ and $\ell$ be positive integers with $n \ge 2$ and $P_1, \ldots, P_n$ be polynomials that map $\mathbb{Z}$ to $\mathbb{Z}$ and have positive leading coefficient. 

Let $D(t) = F_{m,\ell}(P_1(t), \ldots, P_n(t))$ and $E(t) = G_{m}(P_1(t), \ldots, P_n(t))$ for all integers $t$ such that $P_1(t), \ldots, P_n(t)$ are positive; this holds for sufficiently large $t.$ Then $D$ and $E$ are EQP.
\end{theorem}

As noted by Roune and Woods, when $P_1, \ldots, P_n$ are polynomials of two integer parameters, $F(P_1, \ldots, P_n)$ is a function of two parameters which is generally not considered eventually quasi-polynomial. For example, $F(s, t)=\text{lcm}(s,t)-s-t$ is not considered quasi-polynomial; it lacks a periodic structure as $(s,t)$ ranges over the integer lattice.
\subsection{Parametric Integer Linear Programming}

The parametric Frobenius problem has a particular parametric structure that motivates us to examine a more general parametric problem.  In this section, we introduce parametric integer linear programs and another type of parametric problem. The latter type is a generalization of the parametric Frobenius problem, and we show that related functions are eventually quasi-polynomial.

An integer program is the optimization of a certain objective function over the integers subject to certain constraints. Often in integer programming, the constraints and objective functions are linear functions of the indeterminates. This is known as integer linear programming.

Suppose that the indeterminates are $\mathbf{x}=(x_1, \ldots, x_n).$ A program is in canonical form if the objective function is $ \mathbf{c}^{\intercal} \mathbf{x}$ for $\mathbf{c} \in \mathbb{Z}^n,$ we are trying to maximize the objective, and the constraints are $x_i  \in \mathbb{Z}_{\ge 0}$ for all $i$ and $A\mathbf{x} \le \mathbf{b}$ for $A \in \mathbb{Z}^{m \times n}$ and $b \in \mathbb{Z}^m.$ (In this paper, relations between vectors are coordinate-wise.)

Parametric Integer Linear Programming refers to considering a family of linear integer programs parametrized by a variable $t$ i.e. the coefficients of the objective and/or constraints are functions of $t.$ The optimum value of the objective function is a function of $t,$ (which we call the optimum value function) which leads us to questions about this function. Examples for the domain of $t$ are the interval $[0,1]$ or the positive integers. There are many algorithmic results on parametric integer linear programs (PILP) but few theoretical results the properties of the optimum value function or other properties of a PILP. 

A type of PILP relevant to the parametric Frobenius problem is the case when all coefficients are integer polynomials of $t$ and $t$ is an integer parameter.

We use two known results about such PILPs. The following is Theorem 2.1 of \cite{CLS} by Chen, Li, and Sam.

\begin{theorem} \label{Chen}
Let $n$ and $m$ be positive integers. Let $\mathbf{c}$ be in $\mathbb{Z}[u]^n$, $A$ be in $\mathbb{Z}[u]^{m \times n}$, and $\mathbf{b}$ be in $\mathbb{Z}[u]^m$

For all $t$, let 
\begin{displaymath} R(t) := \{ \mathbf{x} \in \mathbb{R}^n \mid A(t) \mathbf{x} \le \mathbf{b}(t) \},\end{displaymath}

the set of real vectors that satisfy a certain conjunction of inequalities. Let $L(t)$ be the set of lattice points in $R(t).$ Assume that $R(t)$ is bounded for all $t.$ Let $g(t)=|L(t)|.$ Then $g$ is eventually quasi-polynomial.
\end{theorem}

A critical tool used in the proof of this theorem is base $t$ representations, which reduces the theorem to the case of $A$ in $\mathbb{Z}^{m \times n}.$ Shen used the same idea to prove the following result, which concerns the optimization of an objective function rather than counting points.

\begin{theorem}[{Corollary 3.2.2 in \cite{Shen}}] \label{Shen}
Let $n$ and $m$ be positive integers. Let $\mathbf{c}$ be in $\mathbb{Z}[u]^n$, $A$ be in $\mathbb{Z}[u]^{m \times n}$, and $\mathbf{b}$ be in $\mathbb{Z}[u]^m$

For all $t$, let 
\begin{displaymath} R(t) := \{ \mathbf{x} \in \mathbb{R}^n \mid A(t) \mathbf{x} \le \mathbf{b}(t) \},\end{displaymath}

the set of real vectors that satisfy a certain conjunction of inequalities. Let $L(t)$ be the set of lattice points in $R(t).$ Assume that $R(t)$ is bounded for all $t.$ For all positive integers $\ell,$ Let $f_\ell(t)$ be the $\ell^{\text{th}}$ largest value of $\mathbf{c}^{\intercal}(t) \mathbf{x}$ with multiplicity for $\mathbf{x}$ in $L(t)$ or $-\infty$ if $|L(t)|<\ell$. Then for all $\ell,$ $f_\ell$ is eventually quasi-polynomial.
\end{theorem}

The constant function at $-\infty$ is to be interpreted as a polynomial. By $\ell^{\text{th}}$ largest value with multiplicity, we mean the $\ell^{\text{th}}$ largest value in the multiset formed by evaluating the objective function on the set $L(t).$

\begin{comment} A PILP in the form shown in \ref{Shen} is in \textit{canonical form}. A PILP in \textit{standard form} is the same, except ``$A(t) \mathbf{x} \le \mathbf{b}(t)$" is replaced by ``$A(t) \mathbf{x} = \mathbf{b}(t)$." The two forms are closely related. A PILP in standard form is already in canonical form because one equality can be written as two inequalities. A PILP in canonical form, $Q$, can be written in standard form by introducing slack variables (extra indeterminates). The objective function is the same, ignoring the slack variables. If the real vector set $R(t)$ for $Q$ is bounded for all $t,$, then the real vector set for the PILP in standard form is bounded for all $t$ too. F/or each $t,$ the lattice point sets for the two PILPs are in bijection, and evaluating the respective objective function commutes with the bijection. See Section 2 of \ref{Shen}. Theorems \ref{Chen} and \ref{Shen} are also true for PILPs in standard form. \end{comment}

An derivative of PILPs that is the focus of this paper involves excluding the projection of one parametric linearly defined set of integers from another parametric linearly defined set. We call this a parametric exclusion problem

\begin{theorem}[Main Theorem] \label{main2} 

Let $m, n_1, n_2, p_1, p_2$ be positive integers. Let $\mathbf{c}$ be in $\mathbb{Z}[u]^{n_2},$ $A_1$ be in $\mathbb{Z}[u]^{p_1 \times (n_1 + n_2)},$ $A_2$ be in $\mathbb{Z}[u]^{p_2 \times n_2},$ $\mathbf{b_1}$ be in $\mathbb{Z}[u]^{p_1},$ and $\mathbf{b_2}$ be in $\mathbb{Z}[u]^{p_2}.$ Let the indeterminates be $x_1, \ldots, x_{n_1+n_2}.$ Let $\mathbf{x_1}=(x_1, \ldots, x_{n_1+n_2}), \mathbf{x_2}=(x_1, \ldots, x_{n_2}),$ and $\mathbf{x_3}=(x_{n_2+1}, \ldots, x_{n_1+n_2}).$

For all $t,$ let 
\begin{displaymath} R_1(t):=\{\mathbf{x_1} \in \mathbb{R}^{n_1+n_2} \mid \mathbf{x_1} \ge \mathbf{0} \wedge A_1(t) \mathbf{x_1} = \mathbf{b_1}(t)\} \end{displaymath}

and 
\begin{displaymath} R_2(t):=\{\mathbf{x_2} \in \mathbb{R}^{n_2} \mid \mathbf{x_2} \ge \mathbf{0} \wedge A_2(t) \mathbf{x_2} \le \mathbf{b_2}(t)\}. \end{displaymath}

Assume that $R_1(t)$ and $R_2(t)$ are bounded for all $t.$  Let $L_1(t):=R_1(t) \cap \mathbb{Z}^{n_1+n_2}, L_2(t):=\\ R_2(t) \cap \mathbb{Z}^{n_2},$ the set of lattice points in $R_1(t)$ and $R_2(t),$ respectively. 

Let 
\begin{displaymath} L_3(t):=\left \{ \mathbf{x_2} \in L_2(t) \mid  \# \left\{ \mathbf{x_3} \in \mathbb{Z}^{n_1} \mid (\mathbf{x_2} \oplus \mathbf{x_3}) \in L'_1(t) \right \}  < m  \right \},\end{displaymath}
or the set of $\mathbf{x_2}$ in $L_2(t)$ which is the image of less than $m$ points when projecting $L_1(t)$ onto $\mathbb{R}^{n_2}$ (using the last $n_2$ coordinates).

For all (positive integers) $t$ and $\ell,$ let $f_\ell(t)$ be $-\infty$ if $|L_3(t)|<\ell$ or the $\ell^{\text{th}}$ largest value with multiplicity of $\mathbf{c}^{\intercal}(t) \mathbf{x_2}$ for $\mathbf{x_2}$ in $L_3(t)$ otherwise, and let $g(t)$ be $|L_3(t)|.$ Then for all $\ell,$ $f_\ell$ and $g$ are eventually quasi-polynomial.
\end{theorem}

In this paper, $\mathbf{v_1} \oplus \mathbf{v_2}$ means the vector which has the coordinates of $\mathbf{v_1}$ first then those of $\mathbf{v_2}$. 

As is the case for PILPs, one could formulate different versions of Theorem \ref{main2} by switching to $A_1(t) \mathbf{x_1} \le \mathbf{b_1}(t)$ and/or $A_2(t) \mathbf{x_2} = \mathbf{b_2}(t),$ but we do not do so here. The inequality conventions here are chosen to resemble the Frobenius problem. We say that the parametric exclusion problem shown in \ref{main2} is given by positive integers $m, n_1, n_2,$ etc. We call $\mathbf{c}^{\intercal}(t) \mathbf{x_2}$ the objective function, $\{f_\ell\}_{\ell}$ or $\{f_\ell\}$ the family of optimum value functions and $g$ the size function.

Later, we show that Theorem \ref{main} follows easily from the above theorem.

\subsection{Notation and Terminology}
Relations between vectors are coordinate-wise. The variables $t, \ell,$ and $m$ are always positive integers. The phrase $t \gg 0$ means ``for sufficiently large t." The constant function at $-\infty$ is to be interpreted as a polynomial. PILP means ``parametric integer linear program." EQP means ``eventually quasi-polynomial" or ``eventual quasi-polynomial." See Definition \ref{deqp} for the meaning of the period and components of an EQP. By the $\ell^{\text{th}}$ largest value, we mean the $\ell^{\text{th}}$ largest value with multiplicity. The symbol $\vee$ denotes logical disjunction, $\wedge$ denotes logical conjunction, and either may be used as a binary operator or with an indexing set. The cardinality of a set $S$ is denoted by $\# S.$ 

When we refer to elements of $\mathbb{R}^m$ or an expression like $(v_1, \ldots, v_m),$ we mean a (column) vector, especially in matrix operations. In general, $\mathbf{x_1}$ denotes the vector $(x_1, \ldots, x_{n_1+n_2}), \mathbf{x_2}=(x_1, \ldots, x_{n_2}),$ and $\mathbf{x_3}=(x_{n_2+1}, \ldots, x_{n_1+n_2}),$ and the same is true for other letters. The symbol $\oplus$ as in $\mathbf{v_1} \oplus \mathbf{v_2}$ means concatenation of vectors. 

A parametric inequality specifically has the form $\mathbf{a}^{\intercal}(t) \mathbf{z} \le b(t),$ where $\mathbf{a}$ is a $\mathbb{Z}[u]$-vector of the correct dimension, $b$ is in $\mathbb{Z}[u],$ and $\mathbf{z}$ refers to a vector of \textit{integer} indeterminates e.g. $\mathbf{x_1}.$ A parametric equation is the same with equality instead of ``$\le$." 

We say that the parametric exclusion problem shown in Theorem \ref{main2} is given by positive integers $m, n_1, n_2,$ etc. We call $\mathbf{c}^{\intercal}(t) \mathbf{x_2}$ the objective function, $\{f_\ell\}_{\ell}$ or $\{f_\ell\}$ the family of optimum value functions and $g$ the size function.

\section{Some Examples of Parametric Exclusion}
\label{sec-ex}
In this section, we discuss concrete examples of parametric exclusion and ideas which are used to simplify these examples. These ideas are applicable to the main problem. 

\textbf{Example 1.} We consider a continuous (real) version of the parametric exclusion problem in which the parameter $t$ is still a positive integer. Let
\begin{align*}R_1(t) &:= \{(x_1, x_2) \mid t/2 \le x_1 + x_2 \le t \wedge t/2 \le x_1-x_2 \le t/2 \}, \\
R_2(t) &:= \{x_2 \mid 0 \le x_2 \le t\}, \\
R_3(t) &:= \{x_2 \in R_2(t) \mid \{x_1 \in \mathbb{R} \mid (x_1, x_2) \in R_1(t)\} = \varnothing\}.
\end{align*}

For fixed $t$, we are projecting a two-dimensional region onto a one-dimensional region. The image of the two-dimensional region is the same as the image of its boundary. The boundary is the union of finitely many one-dimensional regions, which are simpler to understand. Moreover, these edges change in a ``polynomial way'' with respect to $t.$ To be specific, the edges of $R_2(t)$ are 
\begin{align*}(0, t/2) - (t/2, 0), \> (t/2, 0) - (3t/4, t/4), \> (3t/4, t/4) - (t/4, 3t/4), \> (t/4, 3t/4) - (0, t/2).
\end{align*}

Their projections onto the $c_2$ coordinate are $[0, t/2], [0, t/4], [t/4, 3t/4], [t/2, 3t/4],$ respectively. The set $R_3(t) = (3t/4, t]$ for all $t.$ Maximizing an objective function over this region of the parametric type in Theorem \ref{main2} is straightforward.

\textbf{Example 2.} We consider a higher-dimensional version of the previous problem. Let 
\begin{align*}R_1(t) &:= \{(x_1, x_2, x_3) \mid (x+y+z \ge 0) \wedge (2x+y-z\ge t) \wedge (x-y+2z \ge t) \wedge (-x+2y+z\ge t)\}, \\
R_2(t) &:= \{(x_2,x_3) \mid 0 \le x_2, x_3 \le t\}, \\
R_3(t) &:= \{(x_2, x_3) \in R_2(t) \mid \{x_1 \in \mathbb{R} \mid (x_1, x_2, x_3) \in R_1(t)\} = \varnothing\}.
\end{align*}

One can check that $R_1(t)$ is a tetrahedron with vertices $(t/2, t/2, t/2), (5t/7, -4t/7, -t/7),$\\$ (-4t/7, -t/7, 5t/7),$ and $(-t/7, 5t/7, -4t/7).$ 

We follow the same pattern as the previous example. The projection of a 3-dimensional polyhedron onto two dimensions equals the projection of its boundary. The boundary equals four triangles. The projection of each triangle is also a triangle. We are interested in the region $R_3(t),$ or points in $R_2(t)$ that lie outside of the projections of four triangles. 

One way to determine if a point $(x_2, x_3)$ in a plane lies in the closed triangle with vertices $v_1, v_2, v_3$ is to write $(x_2, x_3)$ as an affine linear combination of the vertices i.e.
\begin{align*}(x_2, x_3) = a_1v_1 + a_2v_2 + a_3v_3, a_1+a_2+a_3=1.\end{align*}

Then $(x_2, x_3)$ lies in the triangle iff $0 \le a_1, a_2, a_3 \le 1.$ The constants $a_1, a_2, a_3$ can be written in terms of $x_2, x_3, v_1, v_2,$ and $v_3.$ 

We are dealing with a parametric version of the problem, so we want to know when a point $(x_2, x_3)$ lies in a triangle as $t$ changes. It turns out that this condition can be written as a finite logical combination of linear inequalities in terms of $x_2, x_3,$ and $t.$ Similarly, the set $R_3(t)$ can be written as a finite logical combination of linear inequalities of $x_2, x_3,$ and $t.$

\textbf{Example 3.} We consider another continuous parametric exclusion problem in which $R_1(t)$ is as in Example $2$ and
\begin{align*}
R_2(t) &:= \{(x_3) \mid 0 \le x_3 \le t\}, \\
R_3(t) &:= \{(x_3) \in R_2(t) \mid \{(x_1, x_2) \in \mathbb{R}^2 \mid (x_1, x_2, x_3) \in R_1(t)\} = \varnothing\}.
\end{align*}

We are now projecting a $3$-D space onto a $1$-Dl space. The projection of a $3$-D polyhedron equals the projection of its boundary, but each face is $2$-D. As an example 1, when we are projecting a $2$-D region onto a $1$-D space, it is simpler to think of just projecting the boundaries of the $2$-D region (even if the $2$-D region lies in a $3$-dimensional ambient space). The projection of a $1$-dimensional region onto a $1$-D space is simple. The polyhedron has finitely many faces, and each face has finitely many edges, so we can reduce the description of $R_3(t)$ into a finite logical combination of linear inequalities in terms of $x_3$ and $t.$

\textbf{Example 4.} Consider a continuous parametric exclusion problem in which $n_1, n_2>1,$ $R_1(t)$ lies in $n_1+n_2$ dimensions, $R_2(t)$ lies in $n_2$ dimensions, and 
\begin{align*}R_3(t):= \{ \mathbf{x_2} \in \mathbb{R}^{n_2} \mid \{\mathbf{x_3} \in \mathbb{R}^{n_1} \mid (\mathbf{x_2} \oplus \mathbf{x_3}) \in R_1(t)\} = \varnothing\}.\end{align*}

We are projecting an $n_1+n_2$-D polytope onto an $n_2$-D space. Ideally, we want to reduce this projection problem to the projection of $n_2$-dimensional simplices. This is because the projection of a simplex is usually a simplex. We can determine when a point $\mathbf{x_2}$ is in an $n_2$-D simplex in $n_2$ dimensions by expressing $\mathbf{x_2}$ as an affine linear combination of the vertices and checking if all of the coefficients lie in $[0, 1].$ It turns out that we can ignore simplices whose projections are degenerate (i.e. not simplices). 

This reduction is possible because of the following facts. The projection of an $m$-D polytope onto an $n_2$-D space where $m>n_2$ equals the projection of the boundary of that polytope onto that $n_2$-D space. A polytope has finitely many polytopes as faces (meaning codimension $1$), iteratively taking boundaries leads to a finite number of $n_2$-dimensional polytopes. A polytope $P$ is the union of all simplices whose vertices are a subset of $P.$ 

Thus $R_3(t)$ can be written as a finite logical comination of linear inequalities in $\mathbf{x_2}$ and $t.$

There is a caveat when the projection of $R_1(t)$ onto the coordinates $\mathbf{x_2}$ does not affinely span the whole space. In this case, we should identify the affine subspace spanned by the projection of $R_1(t).$ This subspace depends on $t$ in a ``polynomial'' way. All points outside of this affine subspace but in $R_2(t)$ are definitely in $R_3(t).$ For points in this subspace, we can proceed similarly as in this example. 

Here is a more concrete argument about why we should distinguish the affine subspace. Suppose that we have four non-coplanar points $\mathbf{v_i}$ in $\mathbb{R}^6,$ and we want to know when another point $\mathbf{v}$ is in the simplex defined by these four points. One way to figure this out is to first figure out if $\mathbf{v}$ is in the affine span of the $\mathbf{v_i}.$ If not, then the answer is obviously no. If so, then we can write $\mathbf{v}$ as an affine linear combination of the $\mathbf{v_i}$ uniquely. Then the answer is yes iff all coefficients are in $[0, 1].$ 

Strictly speaking, one doesn't need to separate these steps. One can find a $3$-D subspace such that the $\mathbf{v_i}$ project to non-coplanar points. The projection of $\mathbf{v}$ can be written as an affine linear combination of the projections of the $\mathbf{v_i}$ uniquely. The answer is yes iff all coefficients are in $[0, 1]$ and $\mathbf{v}$ equals this linear combinatino of the $\mathbf{v_i}.$ However, we are emphasizing the first method because that is the pattern used in the general proof.

The set $J(\alpha, t)$ described in Section \ref{sec-ep} is equal to the affine subspace spanned by the projection of $L_1(t)$ (for $t \gg 0.$) In the most interesting cases, $J(\alpha, t)$ is usually the whole space, but we introduce this affine subspace for completeness.

Henceforth, we will not discuss the continuous parametric exclusion problem in detail.

\textbf{Example 5.} Consider a parametric exclusion problem given by
\begin{align*}L_1(t) &:= \{(x_1, x_2) \in \mathbb{Z}^2 \mid (x_2 \le t) \wedge (3/5 x_2 \le x_1 \le 5/8 x_2)\}, \\
L_2(t) &:= \{x_2 \in \mathbb{Z} \mid 0 \le x_2 \le t\}, \\
L_3(t) &:= \{x_2 \in L_2(t) \mid \# \{x_1 \in \mathbb{Z} \mid (x_1, x_2) \in L_1(t)\} < m\}.
\end{align*}
where $m=1.$ We define $R_1(t)$ in the obvious way. $R_1(t)$ is a triangle. The projection of its boundary equals $[0, t].$ However, the projection of $L_1(t)$ does not equal $0, \ldots, t.$ In particular, there is no integer $x_1$ such that $(x_1, 1) \in L_1(t)$ because there is no integer $x_1$ such that $3/5 \le x_1 \le 5/8.$ 

One of the edges of $R_1(t)$ is from $(0, 0)$ to $(t, 3t/5).$ Its projection is $[0, t],$ but the projection of the lattice points on this edge is the set of multiples of $5.$ This does not give us information about when non-multiples of $5$ are in $L_3(t).$ To study non-multiples of $5,$ we denote the line containing these edge by $p,$ and for $i=1, 2, 3, 4,$ we denote the translation of $p$ by $i/5$ in the $x_1$ direction by $p_i.$

 All integers $a$ in $[0, t]$ are the $x_2$-coordinate of exactly one lattice point on these five lines. This lattice point also turns out to be the lattice point in $L_1(t)$ with $x_2 = a$ and minimal $x_1$ coordinate whenever $a$ is not in $L_3(t).$ In other words, $a$ is in $L_3(t)$ iff the unique lattice point on $p, p_1, p_2, p_3, p_4$ with $x_2=a$ lies in $L_1(t).$ It remains to see if this lattice point lies in $L_1(t).$ 

A full resolution of this example is too long for this section.

\textbf{Example 6.} Consider the same problem as in Example $5$ except $m=2.$ We define the line $p$ as before and $p_5, \ldots, p_9$ similarly to $p_1, \ldots, p_4.$ Then for $a=0, \ldots, t,$ there exist exactly two total lattice points with $x_2=a$ on these ten lines. These also turn out to be the two lattice points in $L_1(t)$ with minimal $x_2=a$ and minimal $x_1$ coordinate, whenever $z$ is not in $L_3(t).$ It remains to check if both of these points lie in $L_1(t).$

\textbf{General Problem.} Consider a parametric exclusion problem in which the matrix $A_1$ does not have constant entries. (Note that in all of the previous examples, $A_1$ and $A_2$ had constant entries.) Now, the region $R_3(t)$ may be a polytope whose faces have ``slopes'' that change as a function of $t.$ Therefore, we probably cannot extend the argument relating to $p, p_1, p_2, p_3, p_4$ in example $5$ to the general problem directly. This is because we would need a polynomial number of parallel hyperplanes, and we can say much less about PILPs with a polynomial number of constraints. Therefore, our argument first reduces the parametric exclusion problem to an exclusion-type problem in which the matrix has constant entries.

\section{Outline of Proof Method for an Example of the Parametric Frobenius Problem}
\label{sec-frob-ex}
In this section, we discuss a concrete example of the parametric Frobenius problem. The argument in this section follows a different order than the general proof, but it should offer some insight into the proof that the parametric exclusion problem is EQP. In this section, we are taking $t$ to be ``sufficiently large.''

Let $P_1(t) = t, P_2(t) = t^2+1, P_3(t) = t^2+2t-1.$ Let $D(t) = F(P_1(t), P_2(t), P_3(t)).$ One can prove that $0 \le D(t) < t^3$ (for $t \gg 0$). Therefore, $D(t)$ can be written as a $3$-digit number in base $t$ (with possible leading zeroes). Also, for all $t$ and integers $x$ less than $t^3,$ the only possible nonnegative linear combinations 
\begin{align*}x = a_1 P_1(t) + a_2 P_2(t) + a_3 P_3(t) \end{align*}
have $0 \le a_1 < t^2$ and $0 \le a_2, a_3 < t.$ Therefore, we can rephrase this instance of the parametric Frobenius problem as follows:
\\

For all $t,$ $D(t)$ is the maximum value of $c_2 t^2 + c_1 t + c_0$ subject to the constraints $c_0, c_1, c_2 \in \{0, \ldots, t-1\}$ and there do not exist $b_{11}, b_{10}, b_2, b_3$ in $\{0, \ldots, t-1\}$ such that
\begin{align}\label{ex} (b_{11} t + b_{10})P_1(t) + b_2 P_2(t)+b_3P_3(t) = c_2 t^2 + c_1 t + c_0.\end{align}

A priori, it is not obvious why we want to consider the digits separately. The reason is roughly explained after equation $($\ref{add_exp}$)$.

Let $\mathbf{c} = (c_0, c_1, c_2)$ and $\mathbb{P}(\mathbf{c}, t)$ be the logical proposition that for all $b_{11}, b_{10}, b_2, b_3$ in $\{0, \ldots, t-1\},$ $(b_{11} t + b_{10})P_1(t) + b_2 P_2(t)+b_3P_3(t) \neq c_2 t^2 + c_1 t + c_0.$ We want to find a simpler description of the logical proposition $\mathbb{P}(\mathbf{c}, t).$ 

Our general strategy is to try to solve for the $b_i$ on the left hand side of equation $($\ref{ex}$)$ in terms of $\mathbf{c}$ and $t.$ We can interpret the left hand side as having three powers of $t,$ namely $1, t, t^2,$ with coefficients in $\{0, \ldots, t-1\}.$ Roughly speaking, the equation $($\ref{ex}$)$ is slightly weaker than having three linear equations in four unknowns $b_i,$ and we want to eliminate one of the unknowns.

Since there are only three powers of $t$ on the left hand side and four variables, there is a $\mathbb{Q}$-linear dependency i.e. a nonzero integer vector $(d_{11}, d_{10}, d_2, d_3)$ such that adding this vector to $(b_{11}, b_{10}, b_2, b_3)$ leaves the left hand side unchanged. In this case, one such vector is $(-2, -2, 1, 1)$ because $(-2t-2)P_0(t)+(1)P_1(t)+(1)P_2(t)=0.$

If $\mathbb{P}(\mathbf{c}, t)$ is false, then this invariant translation shows that there exist $b_{11}, b_{10}, b_2, b_3$ in $\{0, \ldots, t-1\}$ satisfying the equation $($\ref{ex}$)$ such that $b_{11} < 2, b_{10} < 2, b_2 > t-2,$ or $b_3 >t-2.$ (This is because we can repeatedly translate to get another quadruple that satisfies the equation; the last translation before leaving $\{0, \ldots, t-1\}^4$ satisfies one of these four inequalities.)

Let $\mathbb{P}_{b_{11} = 0}(\mathbf{c}, t)$ be the logical proposition that there exist $b_{11}, b_{10}, b_2, b_3$ in $\{0, \ldots, t-1\}$ such that $b_{11} = 0$ and $(b_{11} t + b_{10})P_1(t) + b_2 P_2(t)+b_3P_3(t) = c_2 t^2 + c_1 t + c_0.$ Define $\mathbb{P}_{b_{11} = 1},\mathbb{P}_{b_{10} = 0},\mathbb{P}_{b_{10} = 1},\mathbb{P}_{b_{2} = t-1},$ \\ $\mathbb{P}_{b_{3} = t-1}$ similarly. Then the previous paragraph implies the following logical equivalence:
\begin{align}\label{equiv1}\mathbb{P}(\mathbf{c}, t) \Leftrightarrow \neg \left(\mathbb{P}_{b_{11} = 0} \vee \mathbb{P}_{b_{11} = 1}\vee\mathbb{P}_{b_{10} = 0}\vee\mathbb{P}_{b_{10} = 1}\vee\mathbb{P}_{b_{2} = t-1}\vee\mathbb{P}_{b_{3} = t-1}\right).
\end{align}

Consider just the case $b_{11} = 1.$ We want to find a simpler description of the logical proposition $\mathbb{P}_{b_{11} = 1}(\mathbf{c}, t).$ When we fix $b_{11} = 1,$ equation $($\ref{ex}$)$ is 
\begin{align}\label{ex2} t^2 + b_{10} t + b_2 (t^2+1) + b_3 (t^2+2t-1) = c_2 t^2 + c_1 t + c_0. \end{align}

There are 3 $b_i$ remaining on the left hand side of the equation, so we can almost solve for the $b_i$ in terms of $\mathbf{c}$ and $t.$ To understand the situation, let us write the addition equation $($\ref{ex}$)$ vertically in base $t$. 
\begin{align*}
\begin{array}{crrr}
&(d_2)&(d_1)&\\
&1&0&0 \\
&0&b_{10}&0\\
&b_2&0&b_2\\
+&b_3&2 b_3&-b_3\\ \hline
&c_2&c_1&c_0
\end{array}
\end{align*}

In this addition problem, $d_2$ and $d_1$ are carries, which are possibly negative. To be precise, 
\begin{align} \label{add_exp} \begin{split}
b_2 - b_3 &= c_0 + d_1 t, \\
 d_1 + b_{10} + 2 b_3 &= c_1 + d_2 t, \\
  d_2 + 1 + b_2 + b_3 &= c_2.\end{split}\end{align}

Since the variables $b_{10}, b_2, b_3, c_2, c_1, c_0$ are constrained to lie in $[0, t),$ we see that $-2 < d_1 < 1$ and $-1 \le d_2 \le 3.$ In particular, there are at most $10$ different possible values of $(d_1, d_2)$ no matter how large $t$ is. Conceptually, we have shown that single equation $($\ref{ex2}$)$ in three unknowns is slightly weaker than having three linear equations in those three unknowns. This fact is the main reason why we consider the base $t$ representations of the $a_i$ introduced near the beginning of this section. By constraining the digits to the range $[0, t),$ we can constrain the carries to a finite set which is independent of $t.$ If we did not have a finite set of carries (independent of $t$), then conceptually, the single equation $($\ref{ex2}$)$ would be much weaker than having three linear equations in three unknowns.

Let $\mathbb{P}_{b_{11} = 0, d_1 = -1, d_2 = -1}(\mathbf{c}, t)$ be the logical proposition that there exist $b_{11}, b_{10}, b_2, b_3$ in $\{0, \ldots, t-1\}$ such that $b_{11} = 0$ and the equations $($\ref{add_exp}$)$ hold. The previous argument implies the following logical equivalence:
\begin{align}\label{equiv2}\mathbb{P}_{b_{11} = 1}(\mathbf{c}, t) \Leftrightarrow \bigvee_{e_1, e_2 \in S} \mathbb{P}_{b_{11} = 0, d_1 = e_1, d_2 = e_2}(\mathbf{c}, t),
\end{align}
where $S$ is a finite set.

Now consider just the case $b_{11} = 1, d_1 = 0, d_2 = 1.$ We want to find a simpler description of the logical proposition $\mathbb{P}_{b_{11} = 1, d_1 = 0, d_2 = 1}(\mathbf{c}, t).$  Since we have three independent linear equations in three unknowns, we can solve for the $b_i$ in terms of $\mathbf{c}$ and $t.$ The result is
\begin{align}
\label{add_solve}
\begin{split}b_{10} &= c_0 + d_1 t + c_1 + d_2 t - c_2 - d_1 + d_2 + 1, \\
 b_2 &= (c_0 + d_1 t + c_2 - d_2 - 1)/2, \\
 b_3 &= (-c_0 - d_1 t + c_2 - d_2 - 1)/2,\end{split}\end{align}
 where $d_1 = 0$ and $d_2 = 1.$

In particular, these are rational linear combinations of $\mathbf{c}$ and $t.$ Because there is only one solution for $b_{10}, b_2, b_3$ over real numbers, we know that the proposition $\mathbb{P}_{b_{11} = 1, d_1 = 0, d_2 = 1}(\mathbf{c}, t)$ is equivalent to the proposition that this unique solution lies in $\{0, \ldots, t-1\}^3.$ In other words, we have the following logical equivalence:
\begin{align}\label{equiv3}\mathbb{P}_{b_{11} = 1, d_1 = 0, d_2 = 1}(\mathbf{c}, t) \Leftrightarrow 
\left(\begin{array}{ccccc}
b_{10} \ge 0 &\wedge& b_{10} \le t-1 &\wedge& b_{10} \in \mathbb{Z}  \\ 
b_{2} \ge 0 &\wedge &b_{2} \le t-1 &\wedge& b_{2} \in \mathbb{Z}\\
b_{3} \ge 0 &\wedge &b_{3} \le t-1 &\wedge &b_{3} \in \mathbb{Z}\end{array}\right).
\end{align}
Here, $b_{10}, b_2, b_3$ are shorthand for the equations $($\ref{add_solve}$)$. For example, all instances of $b_{10}$ in equation $($\ref{equiv3}$)$ should be replaced by $c_0 + c_1 + t - c_2 + 2.$ It is easy to see that the first two columns on the right hand side can be written as parametric inequalities of the integer indetermintates $\mathbf{c}.$ Unfortunately, the last column cannot. This issue requires introducing new integer indeterminates that act as the remainder term and imposing the constraints that the remainder equals $0.$ 

If we ignore the integrality issue, then we have shown that the proposition $\mathbb{P}_{b_{11} = 1, d_1 = 0, d_2 = 1}(\mathbf{c}, t) $ is equivalent to a finite logical combination of parametric inequalities in $\mathbf{c}.$ If we do not ignore this issue, then we have shown that $\mathbb{P}_{b_{11} = 1, d_1 = 0, d_2 = 1}(\mathbf{c}, t) $ is equivalent to the proposition that $\mathbf{c}$ is in the projection of a PILP given by a finite logical combination of parametric inequalities. This projection is injective, so it is really very straightforward to understand. We do not describe it in detail here.

Together, the equivalences $($\ref{equiv1}$)$ and $($\ref{equiv2}$)$ with $($\ref{equiv3}$)$ prove that the proposition $\mathbb{P}(\mathbf{c}, t)$ is equivalent to a finite logical combination of parametric inequalities in $\mathbf{c}$ (if we ignore the integrality issue). This logical combination involves negations, disjunctions, and conjunctions. One can show that the negation of a parametric inequality is a parametric inequality. Using De Morgan's laws and logical distribution, one can expand this logical expression into a finite disjunction of finite conjunctions of parametric inequalities. 

This expansion tells us that the set of integer vectors $\mathbf{c}$ that satisfy the constraints $0 \le c_i \le t-1$ and $\mathbb{P}(\mathbf{c}, t)$ is a union of finitely many sets $K(\alpha, t)$ of the form in Theorem \ref{Shen}. Here, $\alpha$ is indexed by a finite set which is independent of $t.$

Recall that we want to describe the function $D(t),$ and we showed that for $t>1,$ $D(t)$ is the maximum of the objective function $c_2 t^2 + c_1 t + c_0$ over this feasible set. For all $t,$ this feasible set is a union of finitely many sets of the form $K(\alpha, t).$ For all $t,$ the maximum of an objective function over this union equals the maximum of the maxima of the objective function over the components. By Theorem \ref{Shen}, the maximum of the objective function over $K(\alpha, t)$ is an EQP function of $t.$ One can show that the maximum of finitely many EQP functions is EQP. Therefore, $D(t)$ is EQP. This completes the outline of the proof for this example.

\subsection{Remarks on the general proof}
In general, eliminating just one linear dependency is not enough. We must iteratively eliminate linear dependencies. We keep track of the order in which we eliminate dependencies, so this leads us to a concept of a large rooted tree whose paths from the root to leaves encode information about the order in which we eliminate dependencies. When we finish eliminating linear dependencies, we find that the unknowns on the left hand side are uniquely determined by the analog of $\mathbf{c}$ and $t.$ 

%A major theme in Section \ref{sec-ep} is to prove that the logical proposition $\mathbb{P}(\mathbf{x_2}, t),$ that the lattice point $\mathbf{x_2}$ is the projection of at least $m$ points in $L_1(t),$ is equivalent to a finite logical combination of parametric inequalities, no matter how many are needed. When we eliminate linear dependencies, it is imperative to prove that we do not lose information about when $\mathbb{P}$ is true or false. Many of the logical equivalences are only true for sufficiently large $t.$
%%%%%%%%%%%%%%%%%%%%%%%%%%%%%%%%%%%

Because we are counting preimages in the general problem, not just figuring out when there is at least $1$ preimage, it is helpful to set up the cases so that they are mutually exclusive. In contrast, in equivalence $($\ref{equiv1}$)$ above, the clauses on the right hand side are not mutually exclusive.

%%%%%%%%%%%%%%%%%%%%%%%%%%%% probably omit
%\subsection{Correspondences between ideas in this section and ideas in the general proof}

\section{Base $t$ Representations}
\label{sec-btr}
In this section, we formulate a theorem that resembles Theorem \ref{main2} and show that it implies Theorem \ref{main2}. We prove this new theorem in the next two sections. This theorem is based on the idea of expressing the indeterminates in base $t.$ Despite its appearance, it is much more tractable than Theorem \ref{main2}.

\begin{theorem} \label{main3}Let $m, n_1, n_2, p, q$ be positive integers. Let $\mathbf{c}$ be in $\mathbb{Z}[u]^{n_2},$ $A_1$ be in $\mathbb{Z}^{q \times (n_1 + n_2)}$ (not $\mathbb{Z}[u]^{q \times (n_1+n_2)}$), $A_2$ be in $\mathbb{Z}[u]^{p \times n_2},$ $\mathbf{b_2}$ be in $\mathbb{Z}[u]^{p},$ $\mathfrak{S}$ be a finite set, and $\mathbf{b(\alpha)}$ be in $\mathbb{Z}[u]^q$ for all $\alpha$ in $\mathfrak{S}.$

Let the indeterminates be $y_1, \ldots, y_{n_1+n_2}.$ Let $\mathbf{y_1}=(y_1, \ldots, y_{n_1+n_2}), \mathbf{y_2}=(y_1, \ldots, y_{n_2}),$ and $\mathbf{y_3}=(y_{n_2+1}, \ldots, y_{n_1+n_2}).$

For all $t,$ let 
\begin{displaymath} R_2(t):=\{\mathbf{y_2} \in \mathbb{R}^{n_2} \mid A_2(t) \mathbf{y_2} \le \mathbf{b_2}(t)\}. \end{displaymath}

For all $t$ and $\alpha$ in $\mathfrak{S},$ let
\begin{displaymath} K(\alpha, t):=\{\mathbf{y_1} \in \{0, \ldots, t-1\}^{n_1+n_2} \mid A_1 \mathbf{y_1} = \mathbf{b(\alpha)}(t)\}. \end{displaymath}

Assume, for all $t,$ that $R_2(t)$ is a subset of $[0, t-1]^{n_2}$ and the sets $K(\alpha, t)$ for $\alpha$ in $\mathfrak{S}$ are disjoint. Let $L_2(t)$ be the set of lattice points in $R_2(t).$

Let 
\begin{displaymath} L_3(t):=\left \{ \mathbf{y_2} \in L_2(t) \mid  \# \left\{ \mathbf{y_3} \in \{0, \ldots, t-1\}^{n_1} \mid (\mathbf{y_2} \oplus \mathbf{y_3}) \in \sqcup_{\alpha \in \mathfrak{S}} K(\alpha, t) \right \}  < m  \right \},\end{displaymath}

For all $t$ and $\ell,$ let $f_\ell(t)$ be $-\infty$ if $|L_3(t)|<\ell$ or the $\ell^{\text{th}}$ largest value of $\mathbf{c}^{\intercal}(t) \mathbf{y_2}$ for $\mathbf{y_2}$ in $L_3(t)$ otherwise, and let $g(t)$ be $|L_3(t)|.$ Then for all $\ell,$ $f_\ell$ and $g$ are eventually quasi-polynomial.
\end{theorem}

The idea is that excluding the projection of sets like $K(\alpha, t)$ is much simpler to understand when the matrix has constant entries, even if we must accept having to deal with more than one of these sets. 

\begin{proposition} \label{m3m2} Theorem \ref{main3} implies Theorem \ref{main2} \end{proposition}

To prove this proposition, we use Proposition 3.2 by Shen \cite{Shen}.

\begin{proposition} \label{bounded} Consider a PILP given by a conjunction of parametric inequalities whose real vector set is bounded. Let its lattice point set be $L(t).$ Then there exists a positive integer $r$ such that for $t \gg 0$, $\mathbf{x}$ in $L(t)$ and all $i,$ $0 \le x_i < t^r.$ \end{proposition}
 
Before we prove Proposition \ref{m3m2}, we construct several correspondences between certain lattice point sets and prove that they are bijections. We are essentially thinking of the $x_i$ as $r$-digit integers in base $t,$ treating the digits as indeterminates, and proving that nothing unexpected happens.

Consider a parametric exclusion problem given, as in Theorem \ref{main2}, by positive integers $m, n_1, n_2, \\ p_1, p_2,$ $\mathbf{c}$ in $\mathbb{Z}[u]^{n_2},$ $A_1$ in $\mathbb{Z}[u]^{p_1 \times (n_1 + n_2)},$ $A_2$ in $\mathbb{Z}[u]^{p_2 \times n_2},$ $\mathbf{b_1}$ in $\mathbb{Z}[u]^{p_1},$ and $\mathbf{b_2}$ in $\mathbb{Z}[u]^{p_2}.$ Let its optimum value functions be $\{f_l\}$ and its size function be $g.$

Proposition \ref{bounded} applies to $L_1(t)$ and $L_2(t).$ It tells us that there exists an integer $r$ such that for $t\gg 0$, $\mathbf{x_1}$ in $L_1(t)$ or $\mathbf{x_2}$ in $L_2(t)$ and all appropriate $i,$ $0 \le x_i < t^r.$

We now construct another parametric exclusion problem, $Q'$. Let $m'=m, n_1'=rn_1, n_2'=rn_2.$ Let the indeterminates be $y_{1, 1}, x_{1, 2}, \ldots, y_{n_1+n_2, r}.$ We use the following as shorthand for vectors of these indeterminates: $\mathbf{y_1}=(y_{1, 1}, y_{1, 2}, \ldots, y_{n_1+n_2, r}),$ $ \mathbf{y_2}=(y_{1, 1}, y_{1, 2}, \ldots, y_{n_2, r}),$ and $\mathbf{y_3}=(y_{n_2+1, 1}, y_{n_2+1, 2}, \ldots, x_{n_1+n_2, r}).$ 

Let the objective function be $\sum_{i=1}^{n_2} c_i(t) \sum_{j=1}^r y_{i, j}t^{j-1},$ which can be written as a polynomial covector times $\mathbf{y_2}.$

For all $t,$ let 
\begin{align*} R_1'(t) &:=\left\{ \mathbf{y_1} \in \mathbb{R}^{r(n_1+n_2)} \mid 0 \le y_{i, j} \le t-1 \wedge A_1(t) \left( \sum_{j=1}^r y_{1, j}t^{j-1}, \ldots, \sum_{j=1}^r y_{n_1+n_2, j}t^{j-1}\right) \le \mathbf{b_1}(t)  \right\}, \\ R_2'(t) &:=\left\{ \mathbf{y_2} \in \mathbb{R}^{rn_2} \mid 0 \le y_{i, j} \le t-1 \wedge A_2(t) \left( \sum_{j=1}^r y_{1, j}t^{j-1}, \ldots, \sum_{j=1}^r y_{n_2, j}t^{j-1}\right) = \mathbf{b_2}(t)  \right\}. \end{align*}

Observe that both sets of constraints can be written in terms of parametric inequalities in $\mathbf{y_1}$ and $\mathbf{y_2},$ respectively, which can then be written in matrix form: $A_1'(t) \mathbf{y_2} \le \mathbf{b_1'}(t).$ However, it is convenient to not do so explicitly. Let $L_1'(t)$ and $L_2'(t)$ be the set of lattice points in $R_1'(t)$ and $R_2'(t),$ respectively.

Let 
\begin{displaymath} L_3'(t):=\left \{ \mathbf{y_2} \in L_2'(t) \mid  \# \left\{ \mathbf{y_3} \in \mathbb{Z}^{rn_1} \mid (\mathbf{y_2} \oplus \mathbf{y_3}) \in L'_1(t) \right \}  < m  \right \}.\end{displaymath}

Let $\{f'_\ell\}$ be the optimum value functions and $g'$ be the size function.

Let $\varphi_{1, t}$ be the map from $\mathbb{Z}^{r(n_1+n_2)}$ to $\mathbb{Z}^{n_1+n_2}$ given by 
\begin{displaymath} \varphi_{1,t}(\mathbf{y_1})=\left( \sum_{j=1}^r y_{1, j}t^{j-1}, \ldots, \sum_{j=1}^r y_{n_1+n_2, j}t^{j-1}\right). \end{displaymath}

Define $\varphi_{2, t}$ from $\mathbb{Z}^{rn_2}$ to $\mathbb{Z}^{n_2}$ and $\varphi_{3, t}$ from $\mathbb{Z}^{rn_1}$ to $\mathbb{Z}^{n_1}$ similarly.

\begin{proposition} \label{b1} For $t > N,$ $\varphi_{1, t}$ is a bijection from $L'_1(t)$ to $L_1(t).$ \end{proposition}
\begin{proof}
By construction, $\varphi_{1, t}$ maps $L'_1(t)$ into $L_1(t).$ It is injective because $L_1'(t)$ is a subset of $\{0, \ldots, t-1\}^{r(n_1+n_2)}$ for all $t.$ By the definition of $r,$ for $t > N,$ $\mathbf{x_1}$ in $L_1(t),$ and $0 \le x_1, \ldots, \\ x_{n_1+n_2} < t^r.$  Therefore, for $1 \le i \le n_1+n_2,$ there exist unique $y_{i, 1}, \ldots, y_{i, r}$ in $\{0, \ldots, t-1\}$ such that $x_i = \sum_{j=1}^r y_{i, j} t^{j-1}.$ Then $\mathbf{y_1}=(y_{1, 1}, y_{1, 2}, \ldots, y_{n_1+n_2, r})$ lies in $L_1'(t).$ This shows that $\varphi_{1, t}$ maps $L_1'(t)$ onto $L_1(t),$ as desired. \end{proof}

\begin{proposition} \label{b2} For $t > N,$ $\varphi_{2, t}$ is a bijection from $L'_2(t)$ to $L_2(t).$ \end{proposition} 
The proof is similar.

\begin{proposition} \label{b3} For $t>N$ and $\mathbf{y_2}$ in $\{0, \ldots, t-1\}^{rn_2},$ $\varphi_{3, t}$ is a bijection from \begin{displaymath}  \{\mathbf{y_3} \in \mathbb{Z}^{rn_1} \mid \mathbf{y_2} \oplus \mathbf{y_3} \in L_1'(t)\} \text{  to  }\{\mathbf{x_3} \in \mathbb{Z}^{n_1} \mid \varphi_{2, t}(\mathbf{y_2}) \oplus \mathbf{x_3} \in L_1(t)\}.\end{displaymath}\end{proposition}

\begin{proof} Fix $t>N.$ Since $L_1'(t)$ is a subset of $\{0, \ldots, t-1\}^{rn_2},$ the former set is also a subset of $\\ \{0, \ldots, t-1\}^{rn_1},$ and $\varphi_{3, t}$ maps the former set to its image injectively. 

If $\mathbf{y_3}$ is in the former set, then $\mathbf{y_2} \oplus \mathbf{y_3}$ is in $L_1'(t),$ $\varphi_{1, t}(\mathbf{y_2} \oplus \mathbf{y_3}) = \varphi_{2, t}(\mathbf{y_2})\oplus\varphi_{3, t}(\mathbf{y_3})$ is in $L_1(t)$ (by Proposition \ref{b1}), and $\varphi_{3, t}(\mathbf{y_3})$ is in the latter set.

If $\mathbf{x_3}$ is in the latter set, then $\varphi_{2, t}(\mathbf{y_2}) \oplus \mathbf{x_3}$ is in $L_1(t)$. By the definition of $r,$ $0 \le x_{n_2+1}, \ldots,\\ x_{n_1+n_2} < t^r.$ Therefore, there exists $\mathbf{y_3}$ in $\{0, \ldots, t-1\}^{rn_1}$ such that $\varphi_{3,t}(\mathbf{y_3})=\mathbf{x_3}.$ Then $\varphi_{2, y}(\mathbf{y_2}) \oplus \varphi_{3, t}(\mathbf{y_3}) = \varphi_{1, t}(\mathbf{y_2} \oplus \mathbf{y_3})$ is in $L_1(t),$ $\mathbf{y_2} \oplus \mathbf{y_3}$ is in $L_1'(t)$ (by Proposition \ref{b1}), and $\mathbf{y_3}$ is in the former set, as desired. \end{proof}

\begin{proposition} \label{b4} For $t>N,$ $\varphi_{2, t}$ is a bijection from $L'_3(t)$ to $L_3(t)$ which commutes with evaluating the respective objective function. \end{proposition}
\begin{proof} Fix $t>N.$ $L'_3(t)$ is a subset of $L_2'(t)$, which is a subset of $\{0, \ldots, t-1\}^{rn_2},$ so $\varphi_{2, t}$ maps injectively from $L_3'(t).$ 

If $\mathbf{y_2}$ is in $L_3'(t),$ then $\mathbf{y_2}$ is in $L_2'(t)$ and $ \# \{\mathbf{y_3} \in \mathbb{Z}^{rn_1} \mid \mathbf{y_2} \oplus \mathbf{y_3} \in L_1'(t)\} < m.$ By Propositions \ref{b2} and \ref{b3}, $\varphi_{2, t}(\mathbf{y_2})$ is in $L_2(t),$ and $\# \{\mathbf{x_3} \in \mathbb{Z}^{n_1} \mid \varphi_{2, t}(\mathbf{y_2}) \oplus \mathbf{x_3} \in L_1(t)\} < m,$ so $\varphi_{2, t}(\mathbf{y_2})$ is in $L_3(t).$ 

If $\mathbf{x_2}$ is in $L_3(t)$, then it is in $L_2(t)$ too, and $0 \le x_1, \ldots, x_{n_2}< t^r.$ Therefore, there exists $\mathbf{y_2}$ in $\{0, \ldots, t-1\}^{rn_2}$ such that $\varphi_{2,t}(\mathbf{y_2})=\mathbf{x_2}.$ By Proposition \ref{b2}, $\mathbf{y_2}$ is in $L_2'(t).$ Also from the definition of $L_3(t),$ $ \# \{\mathbf{x_3} \in \mathbb{Z}^{n_1} \mid \varphi_{2, t}(\mathbf{y_2}) \oplus \mathbf{x_3} \in L_1(t)\} < m,$ so by Proposition \ref{b3}, $\\ \# \{\mathbf{y_3} \in \mathbb{Z}^{rn_1} \mid \mathbf{y_2} \oplus \mathbf{y_3} \in L_1'(t)\} < m,$ so $\mathbf{y_2}$ is in $L_3'(t),$ as desired. 

The bijection commutes with evaluating the respective objective function by construction.\end{proof}

As first shown by Chen, Li, and Sam in \cite{CLS}, there exist a finite number of much simpler lattice point sets whose union is $L_1'(t)$ for $t \gg 0.$ Proposition 3.3 of \cite{Shen} slightly modifies their argument and shows the following, which we also use to prove Proposition \ref{m3m2}.

\begin{proposition}[Proposition 3.3 of \cite{Shen}] \label{bt1}
For all $t,$ let 
\begin{displaymath} L_1(t):=\{\mathbf{x_1} \in \mathbb{Z}^{n_1+n_2} \mid \mathbf{x_1} \ge \mathbf{0} \wedge A_1(t) \mathbf{x_1} = \mathbf{b_1}(t)\}. \end{displaymath}

Let $r$ be an integer such that for $t \gg 0,$ $\mathbf{x_1}$ in $L_1(t),$ and all $i,$ $0 \le x_i < t^r.$ Let
\begin{displaymath} L_1'(t):=\left\{ \mathbf{y_1} \in \{0, \ldots, t-1\}^{r(n_1+n_2)} \mid A_1(t) \left( \sum_{j=1}^r y_{1, j}t^{j-1}, \ldots, \sum_{j=1}^r y_{n_1+n_2, j}t^{j-1}\right) \le \mathbf{b_1}(t)  \right\}.\end{displaymath}

In other words, $L_1'(t)$ is the set of $r(n_1+n_2)$-dimensional integer vectors which give the digits of an $(n_1+n_2)$-dimensional integer vector in base $t$ which is in the set $L_1(t).$

Then there exists a finite set $\mathfrak{S},$ a positive integer $q,$ $A$ in $\mathbb{Z}^{q \times r(n_1+n_2)},$ and $\mathbf{b(\alpha)}$ in $\mathbb{Z}[u]^q$ for all $\alpha$ in $\mathfrak{S}$ with the following properties.

For all $t,$ let  
\begin{displaymath} K(\alpha, t):= \left\{ \mathbf{y_1} \in \{0, \ldots, t-1\}^{r(n_1+n_2)} \mid A \mathbf{y_1} = \mathbf{b}(\alpha)(t) \right\}. \end{displaymath}

For $t\gg 0,$ the sets $K(\alpha, t)$ for $\alpha$ in $\mathfrak{S}$ are disjoint, and their union is $L_1'(t).$
\end{proposition}

In simpler terms, Proposition \ref{bt1} is saying the following. Given a PILP in standard form whose lattice point set is a subset of $[0, t^r)^{n_1+n_2},$ we can write all coordinates as $r$-digit integers in base $t.$ In fact, we can form another PILP with lattice point set $L_1'(t)$ such that $L_1'(t)$ is a subset of $[0, t)^{r(n_1+n_2)}$ and $\mathbf{y_1}$ is in $L_1'(t)$ iff the corresponding vector $\mathbf{x_1} \in [0, t^r)^{n_1+n_2}$ lies in $L_1(t).$ 

There exists a finite indexing set $\mathfrak{S}$ (in particular, this indexing set is independent of $t$) such that $L_1(t)$ is the disjoint union of $K(\alpha, t)$ for $\alpha$ in $\mathfrak{S}.$ Each set $K(\alpha, t)$ has a simpler structure. Specifically, it is the lattice point set of a PILP in canonical form whose matrix has constant entries.

\begin{proof}[Proof of Proposition \ref{m3m2}] 
Assume Theorem \ref{main3}. Consider the parametric exclusion problem given by $m, n_1, n_2, p_1, p_2, \mathbf{c}, A_1, A_2, \mathbf{b_1},$ and $\mathbf{b_2}$ with optimum value functions $\{f_\ell\}$ and size function $g.$ Let $r, N$ be integers, and construct another parametric exclusion problem, $Q',$ as before. By construction, for all $t,$ $R'_2(t)$ is a subset of $[0, t-1]^{n_2}.$ By Proposition \ref{bt1}, for $t \gg 0,$ $L_1'(t)$ is the disjoint union of $K(\alpha, t)$ for $\alpha$ in $\mathfrak{S}.$ Therefore, Theorem \ref{main3} applies to the new parametric exclusion problem for $t \gg 0.$ Being EQP only depends on large arguments, so for all $\ell,$ $f'_\ell$ and $g'$ are EQP.

By Proposition \ref{b4}, for $t>N,$ $\varphi_{2, t}$ is a bijection from $L_3'(t)$ to $L_3(t),$ so $g'(t)=g(t)$ (for $t>N.$) By construction, this bijection commutes with evaluating the respective objective function, so for all $\ell,$ $f'_\ell(t)=f_\ell(t).$ Therefore, for all $\ell,$ $f_\ell$ and $g$ are EQP, as desired.
 \end{proof}

% We also have the following easy proposition regarding Theorem \ref{main3}.??????????????????

\section{Excluding Projections}
\label{sec-ep}
%%%%%%%%%%%%%%%%%%%%%%%%%%%%%%%%%%%
In this section, we make the bulk of the argument to prove Theorem \ref{main3}. We show that $L_3(t)$ is closely related to the lattice point set of a PILP whose constraints are a finite disjuction of finite conjunctions of parametric inequalities. This is in contrast to the hypotheses of Theorem \ref{Shen}, which stipulate a finite conjunction of parametric inequalities. Section \ref{sec-untangling} extends Theorem \ref{Shen} to finite disjunctions of finite conjunctions of parametric inequalities.

We use the notation in Theorem \ref{main3}, some of which we repeat here. Let $\mathbf{y_1}=(y_1, \ldots, y_{n_1+n_2}), \mathbf{y_2}=(y_1, \ldots, y_{n_2}),$ and $\mathbf{y_3}=(y_{n_2+1}, \ldots, y_{n_1+n_2}).$ In general, we use the same conventions for vectors such as $\mathbf{z_3}.$ Let $A$ = $A_1,$ a $q \times (n_1+n_2)$ matrix. Recall that 
\begin{displaymath} K(\alpha, t):=\{\mathbf{y_1} \in \{0, \ldots, t-1\}^{n_1+n_2} \mid A_1 \mathbf{y_1} = \mathbf{b(\alpha)}(t)\}. \end{displaymath}

Let 
\begin{displaymath} V(\alpha,t, \mathbf{y_2}):=\left\{\mathbf{y_3} \in \{0, \ldots, t-1\}^{n_1} \mid \mathbf{y_2} \oplus \mathbf{y_3} \in K(\alpha, t) \right\}. \end{displaymath}

The number of elements of $K(\alpha, t)$ which project to $\mathbf{y_2}$ equals $\# V(\alpha, t, \mathbf{y_2})$ because $K(\alpha, t)$ is a subset of $\{0, \ldots, t-1\}^{n_1+n_2}.$

The main idea for this section is that for all $t$ and $\mathbf{y_2}$ in $L_2(t),$ we try to express the proposition $\# V(\alpha, t, \mathbf{y_2})<m$ as a finite logical combination of parametric inequalities in $\mathbf{y_2}$. (It turns out that we require extra indeterminates. Later, we skip right to expressing the proposition $\sum_{\alpha} \# V(\alpha, t, \mathbf{y_2})<m$ as a finite logical combination of parametric inequalities.) 

The first step is to distinguish vectors $\mathbf{y_2}$ are not the projection of \textit{any} real vector in the set
\begin{align} \label{E1} \left\{ \mathbf{y_1} \in \mathbb{R}^{n_1+n_2} \mid A \mathbf{y_1} = \mathbf{b}(\alpha)(t) \right\}. \end{align}%To do this, we try to solve for the other coordinates $\mathbf{y_3}=(y_{n_2+1}, \ldots, y_{n_1+n_2})$ according to $A (\mathbf{y_2} \oplus \mathbf{y_3})= \mathbf{b}(\alpha)(t).$
The reason we want to distinguish these vectors is roughly explained in Example 4 in Section \ref{sec-ex} (specifically, the second half of Example 4). Later, we express this property in terms of parametric inequalities. 

For the second step, we observe that $V(\alpha, t, \mathbf{y_2})$ resembles the solution set of a system of linear equations in $n_1$ variables. The number of linear equations is $q,$ but the number of independent linear equations is $q-q_1,$ where $q_1$ is defined below. We want to count the number of solutions, although counts above $m$ don't matter to us. 

For a few reasons, it is easier to count the number of solutions of $q-q_1$ independent linear equations over a $q-q_1$ dimensional parametric hyperrectangle of lattice points $L$, rather than the $n_1$-dimensional parametric hyperrectangle we're starting with. (Parametric means that this set depends on $t$ in a ``polynomial'' way.)
\begin{itemize}
\item ``Usually,'' the number of solutions is always $0$ or $1.$ It suffices to check non-degeneracy condition.
\item ``Usually,'' there is a unique real vector in the affine span of $L$ which satisfies the linear equations. This unique real vector behaves ``polynomially'' as a function of $t.$ 
\item  The number of solutions equals $1$ iff this unique real vector lies in $L.$
\item Given a real vector in the affine span of $L,$ it is easy to determine of the vector is in $L.$ To do so, we just check that the vector lies in the bounds of the hyperrectangle and that the real vector is a lattice point. 
\end{itemize}

Our argument involves constructing a finite set of $q-q_1$ dimensional parametric hyperrectangles of lattice points which are disjoint subsets of $\{0, \ldots, t-1\}^{n_1}.$ We want to make sure that we don't lose information about $\#V(\alpha, t, \mathbf{y_2})$ when this count is less than $m.$ We prove that for all $m' < m,$ there are exactly $m'$ points in $V(\alpha, t, \mathbf{y_2})$ iff there are exactly $m'$ solutions in the $q-q_1$ dimensional sets. 

The third step is to try to construct a finite logical combination of parametric inequalities that is true if a certain $q-q_1$ dimensional set has exactly $1$ solution and false if that set has $0$ solutions. It turns out that we need extra indeterminates. The reason for this is roughly explained in the discussion after equation $($\ref{equiv3}$).$

The fourth step is to construct an auxiliary PILP $Q^*$ with many extra indeterminates. We prove that the size function and optimum value functions of $Q^*$are equal to the size function and optimum value functions of the parametric exclusion-type problem in Theorem \ref{main3} for $t \gg 0.$ This PILP is not in the form specified by Theorems \ref{Chen} and \ref{Shen}, but it is in almost the correct form.

The final step is to extend Theorems \ref{Chen} and \ref{Shen} by weakening the hypotheses to include $Q^*.$ This last step occurs in the next section.

\subsection{Step 1}
For all $t,$ let
\begin{displaymath} J(\alpha, t) := \{ \mathbf{y_2} \in \mathbb{R}^{n_2} \mid ( \exists \mathbf{y_3} \in \mathbb{R}^{n_1} \mid A (\mathbf{y_2} \oplus \mathbf{y_3})= \mathbf{b}(\alpha)(t)) \}. \end{displaymath}

If $\mathbf{y_2}$ is not in $J(\alpha, t),$ then $V(\alpha, t, \mathbf{y_2})$ is empty. To determine $J(\alpha, t)$, we perform a type of elimination in the variables $\mathbf{y_3}=(y_{n_2+1}, \ldots, y_{n_1+n_2})$ of the matrix $A,$ whose entries are integers. There exists an invertible integer matrix $D$ such that $DA$ has the block form 

\begin{displaymath} \left( \begin{array}{cc} B_1 & 0 \\ B_2 & B_3 \end{array} \right), \end{displaymath} where $B_1$ has $n_2$ columns and $q_1$ rows ($q_1$ depends on $A$) and the rows of $B_3$ are independent. 

Since $D$ is invertible, for all $t,$ $K(\alpha, t)$ equals 
\begin{align} \label{E15} \{y_1 \in \{0, \ldots, t-1\}^{n_1+n_2} \mid DA \mathbf{y_1} = D \mathbf{b(\alpha)}(t)\}. \end{align}

The notation $A, b(\alpha), V(\alpha, t, \mathbf{y_2}), J(\alpha, t), D, B_1, B_2, B_3, q, q_1$ will be used throughout this section.

\begin{proposition} \label{J1} For all $t,$ let $\mathbb{P}(\alpha,t,\mathbf{y_2})$ be the logical proposition  
\begin{displaymath}\bigwedge_{h=1}^{q_1} \left(\left( (-B_1 \mathbf{y_2})_h \le -(D\mathbf{b(\alpha)})_h(t) \right) \wedge \left((B_1 \mathbf{y_2})_h \le (D\mathbf{b(\alpha)})_h(t) \right) \right).
\end{displaymath} 

Then $\mathbf{y_2}$ is in $J(\alpha, t)$ if and only if $\mathbb{P}(\alpha,t, \mathbf{y_2}).$ \end{proposition} 

\begin{proof}
If $\mathbb{P}(\alpha,t, \mathbf{y_2}),$ then $B_1 \mathbf{y_2} = ((D\mathbf{b(\alpha)})_1(t), \ldots, (D\mathbf{b(\alpha)})_{q_1}(t)).$ 

Since $D$ is invertible, $A(\mathbf{y_2} \oplus \mathbf{y_3}) = \mathbf{b(\alpha)}(t)$ if and only if $DA(\mathbf{y_2} \oplus \mathbf{y_3}) = D\mathbf{b(\alpha)}(t),$ and this holds if and only if 
\begin{displaymath} B_2 \mathbf{y_2} + B_3 \mathbf{y_3} = ((D\mathbf{b(\alpha)})_{q_1+1}(t), \ldots, (D\mathbf{b(\alpha)})_{q}(t)).\end{displaymath}

The rows of $\mathbf{B_3}$ are independent, so there exists a real vector $\mathbf{y_3}$ such that the above equation is true, and $\mathbf{y_2}$ is in $J(\alpha, t).$ 

If $\neg \mathbb{P}(\alpha,t, \mathbf{y_2}),$ then $B_1 \mathbf{y_2} \neq ((D\mathbf{b(\alpha)})_1(t), \ldots, (D\mathbf{b(\alpha)})_{q_1}(t))^{\intercal}.$ For any $\mathbf{y_3},$
\begin{displaymath} DA (\mathbf{y_2}+\mathbf{y_3}) =\left( \begin{array}{cc} B_1 & 0 \\ B_2 & B_3 \end{array} \right)(\mathbf{y_2}+\mathbf{y_3})  = (B_1 \mathbf{y_2}) \oplus (B_2 \mathbf{y_2} + B_3 \mathbf{y_3}) \neq D\mathbf{b(\alpha)}(t),\end{displaymath}
so $\mathbf{y_2}$ is not in $J(\alpha, t).$
\end{proof}
Observe that $\mathbb{P}(\alpha,t, \mathbf{y_2})$ is a finite logical combination of parametric inequalities in $\mathbf{y_2}.$ It is possible that $q_1=0$ in which case $\mathbb{P}(\alpha,t, \mathbf{y_2})$ is the empty conjunction, which is always true.

\begin{proposition} The negation of a parametric inequality (by definition in \underline{integer} indeterminates) is a parametric inequality. \end{proposition}

\begin{proof} The negation of $\mathbf{a}^{\intercal}(t) \mathbf{z} \le b(t),$ where $\mathbf{z}$ ranges over integer vectors, $b$ is an integer polynomial, and $\mathbf{a}$ is an integer polynomial vector is the parametric inequality $-\mathbf{a}^{\intercal}(t) \mathbf{z} \le -b(t)-1.$ \end{proof}

\subsection{Step 2}
We recursively define a rooted tree. The construction of this tree is motivated by the equivalence $($\ref{equiv1}$)$, although in this section, we meticulously define the nodes so that the cases are distinct. Each node is associated to a  parametric subset, or just ``subset", of $\{0, \ldots, t-1\}^{n_1}.$ The tree is the same for each $\alpha.$ Each node in the tree is labelled with a function 
\begin{displaymath}H: \{n_2+1 \ldots, n_1+n_2\} \rightarrow \mathbb{Z} \cup ( \mathbb{Z}_{\ge 0} \times \mathbb{Z}_{-}). \end{displaymath}

We use $S(H)$ to denote the node labelled with the function $H,$ and we use $S(H)(t)$ to denote the parametric subset associated with the node $S(H).$ For all $H$, we define
\begin{displaymath} U(H):= H^{-1}(\mathbb{Z}_{\ge 0} \times \mathbb{Z}_{-}), \quad U'(H) := H^{-1}(\mathbb{Z}). \end{displaymath} 
Let $d(H) := \# U'(H),$ the number of elements mapped to integers (rather than ordered pairs). As we will see, $d(H)$ this equals with the depth of $S(H)$ in the tree. We also say that the depth of the subset $S(H)(t)$ equals $d(H).$ The maximum depth is $n_1-q+q_1.$

We now define $S(H)(t)$ for all nodes $S(H).$ For all integers $k,$ let $I(k)$ be the polynomial $k$ if $k \ge 0$ and $u+k$ otherwise. For all $t,$ we define

\begin{align} \label{E2} S(H)(t):= \left \{ \mathbf{y_3} \in \{0, \ldots, t-1\}^{n_1} \left\vert \begin{array}{ll} \left(\forall u \in H^{-1}(\mathbb{Z}), y_u=I(H(u))(t)\right) \, \wedge \\ \left(\forall u \not \in H^{-1}(\mathbb{Z}), H(u)_1 \le y_u \le t+H(u)_2\right) \end{array} \right. \right\}. \end{align}

The set $S(H)(t)$ may be empty for some $H$ and $t.$ Basically, $S(H)(t)$ is a parametric hyperrectangle of lattice points in $\{0, \ldots, t-1\}^{n_1}.$ In each of the $n_1$ dimensions, the projection of $S(H)(t)$ is either one point or an interval of integers $[H(u)_1, t+H(u)_2].$ When the projection is one point, then either this point stays a fixed distance from the lower end of $[0, t-1]$ as $t$ changes \textit{or} this point stays a fixed distance from the upper end of the interval $[0, t-1]$. (This dichotomy is why we introduce the somewhat cumbersome notation $I(k)(t).$) When the projection is an interval, the lower end of the projection stays a fixed distance from the lower end of $[0, t-1]$, and the upper end of the projection stays a fixed distance from the upper end of $[0, t-1].$ 

The root is labeled with the function $H_0$ such that $H_0(u)=(0, -1)$ for all ordered pairs $u$ i.e. $S(H)(t)$ is all of $\{0, \ldots, t-1\}^{n_1}.$

We now define the children of the node $S(H),$ where $d(H)<n_1-q+q_1$. (Nodes $S(H)$ where $d(H) = n_1 - q + q_1$ are leaves.) 

Since $d(H)<n_1-q+q_1,$ $\# U(H) \ge q-q_1+1.$ $B_3$ is an integer matrix with $q-q_1$ rows and columns corresponding to $n_2+1, \ldots, n_1+n_2.$ Consider the columns corresponding to $U(H).$ They are linearly dependent, so there exists a nonzero integer combination of them which is zero, say $\sum_{u \in U(H)} w_{H, u} (B_3)_{i,u} = 0$ for all $i.$ Set $w_{H, u}=0$ for $u$ in $U'(H).$ Define $\mathbf{w_H}$ to be the vector given by the $w_{H, u}$ for $u = n_2+1, \ldots, n_1+n_2.$ (The vector $\mathbf{w_H}$ is an arbitrary choice; to make it determinate, one can take the lexicographically first dependent set and the smallest integer combination whose first nonzero coefficient is positive.) We use $\mathbf{w_H}$ later to refer to this vector. The vector $\mathbf{w_H}$ is indexed similarly to $\mathbf{y_3}$; indeed, our argument uses linear combinations of $\mathbf{y_3}$ and $\mathbf{w_H}.$

The node $S(H)$ has exactly one child for each choice of $v$ in $U(H)$ and an integer $k$ such that $w_{H, v} \neq 0$ and $0 \le k(w_{H, v})^{-1} < m,$ and no other children.  The function $H_{v, k}$ is specified by the following conditions. (For each $u$ in $\{n_2+1, \ldots, n_1 + n_2\},$ $H_{v, k}(u)$ is specified by exactly one of the following cases.) 

\begin{itemize}
\item[\textbf{(a)}] $H_{v, k}(v)=H(v)_1+k$ if $w_{H, v} >0$ and $H_{v, k}(v)=H(v)_2+k$ if $w_{H, v}<0.$

\item[\textbf{(b)}] For all $u$ in $U(H)$ such that $u < v$ and $w_{H, u} \neq 0,$ $H_{v, k}(u)=H(u)+(mw_{H, u}, 0)$ if $w_{H, u}>0$ and $H_{v, k}(u)=H(u)+(0, mw_{H, u})$ if $w_{H, u}<0.$

\item[\textbf{(c)}] $H_{v, k}(u)=H(u)$ for all $u$ in $\{n_2+1, \ldots, n_1+n_2\}$ not covered by \textbf{(b)} and which are not equal to $v.$
\end{itemize}

Observe that $H_{v, k}$ still has the correct form and has 1 higher depth than $H$ (since $H(v)$ is an ordered pair and $H_{v, k}(v)$ is not, and all other types are preserved). Each vertex has finitely many children, and the maximum depth is finite, so this tree is finite.

%Condition \textbf{(b)} is imposed so that the children of a given node are disjoint. If we removed condition \textbf{(b)} and just set $H_{v, k}(u)=H(u)$ for all $u \neq v,$ then Proposition \ref{child} below would still be true, except ``partition'' would be replaced by ``have a union which equals.'' We are meticulously establishing disjointness so that we can exactly count solutions later. 

The notation $H(u), d(H), S(H), S(H)(t), I(k), U(H), U'(H), w_{H, u}, \mathbf{w_H}, H_{v, k}$ and the rooted tree will be used in the rest of this section.
%\begin{remark}The next seven propositions are rather involved, and they form the crux of the argument in this paper. We provide an overview of these propositions and corollaries here. For some of these propositions, we point out the corresponding argument in Section \ref{sec-frob-ex}.\end{remark}
\\

The following two propositions and corollaries prove that the $q-q_1$ dimensional subsets $S(H)(t)$ as $S(H)$ ranges over nodes of maximal depth do not lose any information about $\#V(\alpha, t, \mathbf{y_2})$ when this count is less than $m$

\begin{proposition} \label{child} Fix $t.$ Given $S(H)$ not of maximal depth, the children of $S(H)$ at $t$ partition the set of $\mathbf{y_3}$ in $S(H)(t)$ such that $\mathbf{y_3}-m \mathbf{w_H}$ is not in $S(H)(t),$ where we refer to the coordinates of $\mathbf{y_3}$ and $\mathbf{w_H}$ as $y_{n_2+1}, \ldots, y_{n_1+n_2}.$ \end{proposition}

The proof of this proposition is generally an exercise in manipulating the constraints (\ref{E2}) and conditions $\textbf{(a)}, \textbf{(b)},$ and $\textbf{(c)}.$ One way to understand this proposition, (\ref{E2}), $\textbf{(a)}, \textbf{(b)},$ and $\textbf{(c)}$ is as follows. If $\mathbf{y_3}$ is in $S(H)(t)$ but $\mathbf{y_3} - m \mathbf{w_H}$ is not, then $\mathbf{y_3}$ is close to the boundary of a (parametric) hyperrectangle of lattice points. We can specify finitely many (parametric) sub-hyperrectangles of codimension $1$ whose union contains $\mathbf{y_3}$; moreover, the union exactly equals the possible vectors $\mathbf{y_3}$ such that $\mathbf{y_3} - m \mathbf{w_H}$ is not in $S(H)(t).$ If one explicitly writes out these sub-hyperrectangles, one should arrive at $\textbf{(a)}$ and $\textbf{(c)}.$ 

However, we insist on having disjoint children so that we can exactly count solutions later. Therefore, we look for the smallest index such that $\mathbf{y_3}-m\mathbf{w_H}$ is out of range in the coordinate with that index. Condition \textbf{(b)} is imposed so that the children of a given node are disjoint. If we removed condition \textbf{(b)} and just set $H_{v, k}(u)=H(u)$ for all $u \neq v,$ then a modified version of Proposition \ref{child} below would still be true with ``partition'' replaced by ``have a union which equals.'' 

Many details are painstakingly written out in this proof. The ideas in this proof are not necessary for the rest of this paper. If one finds this proposition obvious, then this proof can safely be skipped.

\begin{proof}[Proof of Proposition \ref{child}]
It's easy to see that the children of $S(H)$ at $t$ are subsets of $S(H)(t).$

First suppose that $\mathbf{y_3}$ and $\mathbf{y_3}-m \mathbf{w_H}$ are in $S(H)(t).$ We check, for all $v$ and $k$ such that $w_{H, v} \neq 0$ and $0 \le k(w_{H, v})^{-1}< m,$ that $\mathbf{y_3}$ is not in $S(H_{v, k})(t).$ If so, then $\mathbf{y_3}$ is not in any child of $S(H)$ at $t$ because these are all the possible children by definition.

\textbf{Case 1:} $w_{H,v} > 0.$ By definition, $w_{H, u} = 0$ for all $u$ in $U'(H),$ so $v$ is in $U(H).$ When $u=v$ in (\ref{E2}), the second line applies. It tells us that all points in $S(H)(t)$ have $v$-coordinate at least $H(v)_1,$ so $y_v, y_v-mw_{H,u} \ge H(v)_1,$ so $y_v \ge mw_{H,v} + H(v)_1.$ Therefore, $\mathbf{y_3}$ is not in $H_{v, k}$ for any $k$ with $0 \le k < m w_{H,v}.$ (Here are more details. Suppose that $\mathbf{y_3}$ is in $H_{v, k}$. Then by the first line of (\ref{E2}) setting $u=v,$ $y_v = I(H_{v, k}(v))(t).$ By condition \textbf{(a)}, $H_{v, k}(v) = H(v)_1+k,$ which is nonnegative, so $I(H_{v, k}(v))(t) = H_{v, k}(v)= H(v)_1+k,$ so $H(v)_1+k = y_v \ge m w_{H, v} + H(v)_1,$ so $k \ge m w_{H, v},$ a contradiction.)

\textbf{Case 2:} $w_{H,v}<0.$ Then $y_v-mw_v \le t+H(v)_2,$ so $y_v \le t+H(v)_2+mw_{H,v},$ and $\mathbf{y_3}$ is not in $H_{v, k}$ for any $k$ with $0 \ge k > m w_{H,v}$. 
\\

Conversely, suppose that $\mathbf{y_3}$ is in $S(H)(t)$ such that $\mathbf{y_3}-m \mathbf{w_H}$ is not in $S(H)(t).$ We want to find $(v, k)$ such that $\mathbf{y_3}$ is in $S(H_{v, k})(t).$ The set $S(H)(t)$ is defined by one clause for each $v$ in $n_2+1, \ldots, n_1+n_2$. Therefore, there exists at least one $v$ with $n_2+1 \le v \le n_1+n_2$ such that $\mathbf{y_3}-m\mathbf{w_H}$ does not satisfy the constraints (\ref{E2}) corresponding to $u=v.$ Let $v$ be the minimal such $v.$ 

Observe that for all $r$ such that $w_{H, r} = 0,$ $(\mathbf{y_3}-m\mathbf{w_H})_{r} = y_{r} = (\mathbf{y_3})_{r}.$ Recall that $\mathbf{y_3}$ is in $S(H)(t),$ so $\mathbf{y_3}$ satisfies the constraint in (\ref{E2}) for $u=r,$ so $\mathbf{y_3}-m\mathbf{w_H}$ does too, hence $v \neq r$ and $w_{H, v} \neq 0.$

\textbf{Case 1:} $w_{H, v}>0$. By definition, $\mathbf{y_3}-m \mathbf{w_H}$ does not satisfy the constraint for $u=v,$ so either $y_v - m w_{H, v} < H(u)_1$ or $y_v - m w_{H, v} > t + H(u)_2.$ On the other hand, $\mathbf{y_3}$ does satisfy the constraint for $u=v$ i.e. $H(u)_1 \le y_v \le t + H(u)_2.$ Since $w_{H, v} > 0,$ we have $y_v - m w_{H, v} < H(u)_1.$ Alternatively, we can write $H(v)_1 \le y_{v} < m w_{v}+ H(v)_1.$ This motivates us to prove that $\mathbf{y_3}$ lies in $S(H_{v, k})(t)$ where $0 \le k:=y_{v}-H(v)_1 < m w_{H,v}.$ 

To show that $\mathbf{y_3}$ lies in $S(H_{v, k})(t),$ we need to show that $\mathbf{y_3}$ satisfies the constraints in (\ref{E2}) for $H_{v, k}$ and $u = n_2+1, \ldots, n_1+n_2.$ To check such constraints, we need to know $H_{v, k}(u).$ Such values are specified by the conditions \textbf{(a)}, \textbf{(b)}, or \textbf{(c)}. Each $u$ in this range is described by exactly one of these conditions. (In these conditions, $H$ and $v$ coincide with the $H$ and $v$ in this proof.)

\textbf{Case 1.(a).} The only $u$ in this case is $u=v.$ Since $H_{v, k}(u)$ is an integer, we need to check $y_{v} \stackrel{}{=} I(H_{v, k}(v))(t).$ This is true by construction.

\textbf{Case 1.(b).} In this case, $u$ is in $U(H),$ $w_{H, u} \neq 0,$ and $u < v.$ By the minimality of $v,$ $\mathbf{y_3}-m\mathbf{w_H}$ satisfies the constraint for $H$ corresponding to $u.$ The vector $\mathbf{y_3}$ satisfies this constraint as well, so $H(u)_1 \le y_u, y_u - m w_{H, u} \le t + H(u)_2.$ 

If $w_{H, u} > 0,$ then by condition \textbf{(b)}, $H_{v, k}(u) = H(u) + (m w_{H, u}, 0).$ We need to check the constraint for $H_{v, k}$ for $u,$ which is $H(u)_1 + m w_{H, u} H_{v, k}(u)_1 \stackrel{}{\le} y_u \stackrel{}{\le} t + H_{v, k}(u)_2 = t + H(u)_2,$ which is immediate from the previous paragraph.

Similarly, if $w_{H, u} < 0,$ then by condition \textbf{(b)}, $H_{v, k}(u) = H(u) + (0, mw_{H, u}).$ We need to check the constraint for $H_{v, k}$ for $u,$ which is $H(u)_1 = H_{v, k}(u)_1 \stackrel{}{\le} y_u \stackrel{}{\le} t + H_{v, k}(u)_2 = t + H(u)_2 + m w_{H, u},$ which is again immediate.

\textbf{Case 1.(c).} In this case, $H_{v, k}(u) = H(u).$ The constraint for $H_{v, k}$ for $u$ is exactly the same as the constraint for $H$ for $u.$ This constraint is satisfied because $\mathbf{y_3}$ is in $S(H)(t).$

%Condition \textbf{(c)} is no stronger than those of $S(H)(t),$ so those are satisfied. \textbf{(a)} is true by construction. For \textbf{(b)}, consider $u'$ such that $w_{H,u'} \neq 0$ and $u' < u.$ $\mathbf{y_3}-m\mathbf{w_H}$ satisfies the constraint corresponding to $u',$ so $H(u')_1 \le y_{u'}, y_{u'}-m w_{H,u'} \le H(u')_2.$ It is easy to check that for $w_{H,u'} >0$ or $w_{u'}<0,$ $H_{u, k}(u')-1 \le y_{u'} \le H_{u, k}(u')_2,$ as desired. 
\textbf{Case 2:} $w_{H, v} < 0.$ This case is similar to Case 1, and we omit the details.
\\

To complete the proof, we show that $\mathbf{y_3}$ as above does not lie in $S(H_{u, k'})$ for any other valid choice of $u$ and $k'.$ We consider three cases based on which of conditions \textbf{(a)}, \textbf{(b)}, or \textbf{(c)} applies to $u.$

\textbf{Case (a).} In this case, $u=v,$ and $k' \neq k.$ Using \textbf{(a)} twice, $H_{v, k}(v) = H(v)_1+k \neq H(v)_1+k' = H_{v, k'}(v).$ Since $\mathbf{y_3}$ is in $S(H_{v, k}(t)),$ $y_v = H_{v, k}(v) \neq H_{v, k'}(v),$ so $\mathbf{y_3}$ does not satisfy the constraint for $H_{v, k'}$ for $u=v.$ 

\textbf{Case (b).} In this case, $u < v,$ $w_{H, u} \neq 0,$ and $0 \le k' (w_{H, u})^{-1} < m.$

\textbf{Case (b).1:} $w_{H, u} > 0.$ By condition \textbf{(b)}, $H_{v, k}(u)=H(u)+(mw_{H, u}, 0).$ Since $\mathbf{y_3}$ is in $S(H_{v, k})(t),$ the constraint for $u$ gives us $H(u)_1+mw_{H, u} \le y_u \le t + H(u)_2.$ By condition \textbf{(a)}, $H_{u, k'}(u)=H(u)_1+k' < H(u)_1+mw_{H, u}.$ Therefore, $\mathbf{y_3}$ does not satisfy the constraint for $S(H_{u, k'})$ for $u.$

\textbf{Case (b).2:} $w_{H, u} < 0.$ By condition \textbf{(b)},$H_{v, k}(u)=H(u)+(0, mw_{H, u}).$ Since $\mathbf{y_3}$ is in $S(H_{v, k})(t),$ the constraint for $u$ gives us $H(u)_1\le y_u \le t + H(u)_2 + mw_{H, u}.$ By condition \textbf{(a)}, $H_{u, k}(u)=H(u)_2+k' > H(u)_2 + m w_{H, u}.$ Therefore, $\mathbf{y_3}$ does not satisfy the constraint for $H_{u, k'}$ for $u.$

\textbf{Case (c).} Recall that $S(H_{u, k'})$ is a node only if $w_{H, u} \neq 0,$ so it remains to check $u$ with $w_{H, u} \neq 0$ and $u>v.$ Also recall that $0 \le k (w_{H, v})^{-1} < m.$

\textbf{Case (c).1:} $w_{H, v} > 0.$ Since $\mathbf{y_3}$ is in $S(H_{v, k}(t)),$ the constraint for $v$ gives us $y_v = I(H_{v, k}(v)+k) = I(H(v)_1+k) = H(v)_1+k.$ By condition \textbf{(b)}, $H_{u, k'}(v)=H(v)+(m w_{H, v}, 0).$ The constraint for $\mathbf{z_3}$ to be in $S(H_{u, k'})(t)$ for $v$ is $H_{u, k'}(v)_1 \le z_v \le y + H_{u, k'}(v)_2,$ which $\mathbf{y_3}$ does not satisfy because $\mathbf{y_v} = H(v)_1+k < H(v)_1+m w_{H, v} = H_{u, k'}(v)_1.$

\textbf{Case (c).2:} $w_{H, v} < 0.$ Since $\mathbf{y_3}$ is in $S(H_{v, k}(t)),$ the constraint for $v$ gives us $y_v = I(H_{v, k}(v)+k) = I(H(v)_2+k) = H(v)_2+k.$ By condition \textbf{(b)}, $H_{u, k'}(v)=H(v)+(0, m w_{H, v}).$ The constraint for $\mathbf{z_3}$ to be in $S(H_{u, k})(t)$ for $v$ is $H_{u, k}(v)_1 \le z_v \le y + H_{u, k}(v)_2,$ which $\mathbf{y_3}$ does not satisfy because $\mathbf{y_v} = H(v)_2+k > H(v)_2+m w_{H, v} = H_{u, k'}(v)_2.$

%Since $H(v')_1 \le y_{v'}-m w_{H,v'} \le H(v')_2,$ $\mathbf{y_3}$ does not lie in $S(H_{v', k'})(t)$ for any $v'$ such that $v' < v$ and $w_{H,v'} \neq 0.$ 

%Suppose that $w_{H,u} > 0.$ Then $H(u)_1 \le y_u < m w_{H,u} + H(u)_1.$ For $u'$ such that $u < u'$ and $w_{u'} \neq 0,$ $S(H_{u', k'})(t)$ has the constraint from \textbf{(b)} that $H(u)_1+mw_{H,u} \le y_u < H(u)_2,$ which is not true for $\mathbf{y_3},$ so $\mathbf{y_3}$ does not lie in $S(H_{u', k'})(t).$ The case for $w_{H,u}<0$ is similar. $S(H_{u', k'})(t)$ only exists when $w_{u'} \neq 0,$ so we have covered all cases. $\mathbf{y_3}$ is in $S(H_{u, k})(t)$ and in no other child of $S(H)$ at $t.$
\end{proof}

\begin{corollary} For all $t,$ the subsets of maximal depth (at $t$) are disjoint. In other words, the sets $S(H)(t)$ as $S(H)$ ranges over nodes of maximal depth are disjoint. \end{corollary}

\begin{remark}One can also check that for any function $H$ of the correct form, there exists $N_H$ such that for all $t>N,$ $S(H)(t)$ is nonempty. Then for all $t$ greater than $N_H$ for all $H$ in the tree, of which there are finitely many, all subsets $S(H)(t)$ of maximal depth are nonempty. This shows that the function labels of all of the leaves of the tree are distinct. This allows us to uniquely refer to each node of maximal depth by its function label.\end{remark}

\begin{proposition} \label{push} Fix $t$ and $\mathbf{y_2}$ in $\mathbb{Z}^{n_2}.$ Assume that less than $m$ elements of $V(\alpha,t, \mathbf{y_2})$ lie in a subset of maximal depth. Then an element of $V(\alpha,t, \mathbf{y_2})$ which lies in any subset $S(H)(t)$ of the tree lies in a subset of maximal depth.
\end{proposition}

\begin{proof}  %inductive hypothesis
We proceed by induction on the depth $d$ of $S(H)$ backwards, from the maximal depth to $0.$ The base case, $d=n_1-q+q_1,$ is trivial. 

Now suppose that $0 \le d<n_1 -q + q_1$ and that the proposition is true for deeper nodes. Assume, for the sake of contradiction, that $\mathbf{z_3}:=(z_{n_2+1}, \ldots, z_{n_1+n_2})$ lies in $V(\alpha,t, \mathbf{y_2})$ and $S(H)(t)$ where $d(H)=d$ but not in any subset of maximal depth. The root contains only lattice points, so $\mathbf{z_3}$ is a lattice point. If $\mathbf{z_3}$ lies in a child of $S(H)(t),$ then we get a contradiction using the induction hypothesis.

Therefore, $\mathbf{z_3}$ does not lie in a child of $S(H)(t).$ let $U(H), U'(H)$ and $\mathbf{w_H}$ be as above. By Proposition \ref{child}, $\mathbf{z_3} + m \mathbf{w_H}$ lies in $S(H)(t).$ It is easy to see that in fact $\mathbf{z_3}+m' \mathbf{w_H}$ lies in $S(H)(t)$ for integers $m'$ with $0 \le m' \le m.$ Consider the sequence of vectors $\{\mathbf{z_3}+m'\mathbf{w_H}\}_{m'=0}^{\infty}.$ This is a sequence of distinct integer vectors in which the first $m+1$ terms lie in $S(H)(t).$ The sequence eventually leaves $S(H)(t)$ because $\mathbf{w_H}$ is nonzero. Therefore, there exists a positive integer $m_0$ such that the terms $m_0, m_0+1, \ldots, m_0+m-1$ all lie in $S(H)(t)$ but $m_0+m$ does not. Since $S(H)(t)$ is a convex set among lattice points, none of the terms after $m_0+m$ lie in $S(H)(t)$ either. By Proposition \ref{child}, the vectors $\mathbf{z_3}+m'\mathbf{w_H}$ for $m'=m_0, \ldots, m_0+m-1$ all lie in a child of $S(H)(t),$ and they are all distinct. Since a child of $S(H)(t)$ has greater depth, the induction hypothesis implies that all of these vectors lie in a subset of maximal depth. 

On the other hand, $\mathbf{z_3}$ lies in $V(\alpha,t, \mathbf{y_2}),$ so $\mathbf{y_2} \oplus \mathbf{z_3}$ lies in $K(\alpha, t).$ By (\ref{E15}),
\begin{displaymath} \left( \begin{array}{cc} B_1 & 0 \\ B_2 & B_3 \end{array} \right) (\mathbf{y_2} \oplus \mathbf{z_3})=D \mathbf{b}(\alpha)(t). \end{displaymath}

Here, $\mathbf{y_2}$ coincides with the first columns of blocks, and $\mathbf{z_3}$ coincides with the second. By the construction of $\mathbf{w_H},$ $B_3 \mathbf{w_H}=0,$ so for all $m',$
\begin{displaymath} \left( \begin{array}{cc} B_1 & 0 \\ B_2 & B_3 \end{array} \right) (\mathbf{y_2} \oplus (\mathbf{z_3}+m' \mathbf{w_H}))=D \mathbf{b}(\alpha)(t). \end{displaymath}

For $m'=m_0, \ldots, m_0+m-1,$ $\mathbf{z_3}+m'\mathbf{w_H}$ is in $S(H)(t)$, which is a subset of $\{0, \ldots, t-1\}^{n_1},$ so $\mathbf{y_2} \oplus (\mathbf{z_3}+m' \mathbf{w_H})$ is in $\{0, \ldots, t-1\}^{n_1+n_2}.$ By (\ref{E15}), $\mathbf{z_3}+m' \mathbf{w_H}$ is in $V(\alpha,t, \mathbf{y_2}).$ This is a contradiction because we assumed that less than $m$ different vectors are in $V(\alpha,t, \mathbf{y_2})$ lie in a subset of maximal depth. Therefore, an element of $V(\alpha,t, \mathbf{y_2})$ which lies in a subset of the tree lies in a subset of maximal depth.
\end{proof}

\begin{corollary} \label{c321} Fix $t.$ For all $\mathbf{y_2}$ and integers $m^*$ with $0 \le m^* < m,$ $\# V(\alpha, t, \mathbf{y_2})=m^*$ if and only if exactly $m^*$ distinct elements of $V(\alpha,t, \mathbf{y_2})$ lie in a subset of maximal depth.
\end{corollary}

\begin{proof} Since the root of the tree is $\{0, \ldots, t-1\}^{n_1},$ the backwards direction is simply the above proposition. For the forwards direction, at most $m^* < m$ elements of $V(\alpha,t, \mathbf{y_2})$ lie in a subset of maximal depth, so the above proposition shows that all elements of $V(\alpha,t, \mathbf{y_2})$ lie in a subset of maximal depth. \end{proof}

\subsection{Step 3}
The following two propositions show that for all $t \gg 0, \alpha$ in $\mathfrak{S},$ $\mathbf{y_2}$ in $J(\alpha, t),$ and $S(H)$ of maximal depth, there exists a unique real vector that lies in the affine span of $S(H)(t)$ and $V_2(\alpha, t, \mathbf{y_2}).$ Moreover, this unique real vector behaves polynomially as a function of $t$ (keeping the other data fixed). The vector of rational polynomials that generates this vector is also constructed.

\begin{proposition} \label{det1} For any node $S(H),$ the columns of $B_3$ corresponding to $U(H)=H^{-1}(\mathbb{Z}_{\ge 0} \times \mathbb{Z}_{-})$ form a rank $q-q_1$ (full rank) matrix. \end{proposition}

\begin{proof} We use induction based on the depth, increasing. The induction hypothesis is that for an integer $d$ with $0 \le d < n_1-q+q_1,$ the proposition is true for all nodes of depth $d.$ Since the rows of $B_3$ are independent, this is true for depth $0$ (the base case).

Suppose that $S(H_{u, k})$ is a child of $S(H)$ and that the proposition is true for $H.$ Let $B_H$ be the matrix of columns of $B_3$ which correspond to $U(H)$ in order and no other columns. Then $B_H$ has full rank. By construction, $U(H_{u, k})=U(H) \setminus \{u\}.$ The $u^{\text{th}}$ column of $B_3$ is dependent on the other columns corresponding to $U(H),$ and all of these columns appear in $B_H,$ and so the removing the corresponding column in $B_H$ leaves a full rank matrix, as desired. \end{proof}

For $\alpha$ in $\mathfrak{S}$ and $S(H)$ of maximal depth, we construct a vector of rational polynomials \\ $\mathbf{Y_3}(\alpha, H)(\mathbf{y_2}, t)$. in other words, all components of $\mathbf{Y_3}(\alpha, H)$ is a polynomial of the variables $\mathbf{y_2}$ and $t.$  Our convention is that $\mathbf{Y_3}(\alpha, H)$ has components $(Y_{n_2+1}, \ldots Y_{n_1+n_2}).$ We now construct the vector $\mathbf{Y_3}(\alpha, H).$ 

Let $\mathbf{Z_3}=(Z_{n_2+1}, \ldots, Z_{n_1+n_2})$ be an auxiliary polynomial vector. For all $u$ in $U'(H),$ let $Y_u$ and $Z_u$ be the polynomial $I(H(t)).$ For $u$ in $U(H)$ (all other $u$ in $\{n_2+1, \ldots, n_1+n_2\}$), let $Z_u$ be $0.$ 

Let $U(H),$ sorted ascending, be $\{v_1, \ldots, v_{|U(H)|}\},$ where $|U(H)|=q-q_1.$ Define $B_H$ as in Proposition \ref{det1}; it is an invertible matrix of integers. With our notation $v_1, \ldots, v_{|U(H)|},$ we can also write $B_H$ as 
\begin{displaymath}\{(B_3)_{q_1+i, v_j}\}_{i=1,j=1}^{q-q_1, q-q_1}.\end{displaymath}

For $h$ from $1$ to $q-q_1,$ let $Y_{v_h}$ be the polynomial
\begin{align*} \label{E3} \left(B_H^{-1} \left(\left(\left(D\mathbf{b(\alpha)}\right)_{q_1+1}, \ldots, \left(D\mathbf{b(\alpha)}\right)_{q} \right) - B_2 \mathbf{y_2} - B_3 \mathbf{Z_3} \right)\right)_h.\end{align*}

It is easy to see that $\mathbf{Y_3}(\alpha, H)$ is a vector, each coordinate of which is a rational polynomial in $\mathbf{y_2}$ and $t.$

\begin{remark}The above construction of $\mathbf{Y_3}(\alpha, H)$ may appear very dense, but it is really not surprising. Conceptually, when we fix $\mathbf{y_2}$ and $n_1-q+q_1$ coordinates of $\mathbf{y_3},$ we have $q-q_1$ unknowns and $q-q_1$ linear equations given by 
\begin{align*}\left(\begin{array}{cc}B_2 & B_3 \end{array}\right)(\mathbf{y_2} \oplus \mathbf{y_3}) = D\mathbf{b}(\alpha)(t).\end{align*}

The columns of $B_3$ that are acting on the unknowns are exactly those corresponding to $U(H).$ By Proposition \ref{det1}, these columns are independent. Therefore, there is a unique rational solution, which changes predictably as a function of $t.$
\end{remark}

\begin{proposition} \label{J2} For all $t,$ $\alpha$ in $\mathfrak{S},$ $\mathbf{y_2}$ in $J(\alpha, t),$ and $S(H)$ of maximal depth, the unique real vector $\mathbf{y_3}$ such that $A (\mathbf{y_2} \oplus \mathbf{y_3}) = \mathbf{b(\alpha)}(t),$ \textit{and} for all $u$ in $U'(H),$ $y_u=I(H(u))(t),$ is given by $\mathbf{y_3}=\mathbf{Y_3}(\alpha, H)(\mathbf{y_2}, t).$ \end{proposition}

\begin{proof} We first check that $\mathbf{y_3}:=\mathbf{Y_3}(\alpha, H)(\mathbf{y_2}, t)$ satisfies the equation. Since $D$ is invertible, it suffices to check
\begin{displaymath} \left( \begin{array}{cc} B_1 & 0 \\ B_2 & B_3 \end{array} \right) (\mathbf{y_2} \oplus \mathbf{Y_3}(\alpha, H)(\mathbf{y_2}, t))=D \mathbf{b}(\alpha)(t). \end{displaymath}

By Proposition \ref{J1}, the first $q_1$ coordinates match. It's easy to see that the equation in the other $q-q_1$ coordinates is given by
\begin{displaymath} B_2 \mathbf{y_2} + B_3 \mathbf{Z_3}(t) + B_H \left( Y_{v_1}(y), \ldots, Y_{v_{|U(h)|}}(t)\right) = \left(\left(D\mathbf{b(\alpha)}\right)_{q_1+1}(t), \ldots, \left(D\mathbf{b(\alpha)}\right)_{q}(t) \right), \end{displaymath} 
and this is true by construction.

We now check that this is unique. If not, we see, by subtraction, that there exists a nonzero vector $\mathbf{y_3'} \in \mathbb{R}^{n_1}$ such that for all $u$ in $U'(H),$ $y_u'=0,$ and 
\begin{displaymath} \left( \begin{array}{cc} B_1 & 0 \\ B_2 & B_3 \end{array} \right) ( \mathbf{0} \oplus \mathbf{y_3'} ) = 0. \end{displaymath}

This contradicts the fact that $B_H$ is invertible, so indeed $\mathbf{Y_3}(\alpha, H)(\mathbf{y_2}, t)$ is the unique vector. 
\end{proof}

\begin{proposition} \label{J3} Given $t, \alpha$ in $\mathfrak{S},$ $\mathbf{y_2}$ in $J(\alpha, t),$ and $H$ of maximal depth, define $\mathbf{Y_3}(\alpha, H)$ as above. Let
\begin{displaymath}\mathbb{P}(\alpha,t, \mathbf{y_2}, H) := \left(\bigwedge_{u=n_2+1}^{n_1+n_2} 0 \le Y_u \le t-1 \right) \wedge \left(\bigwedge_{u \in U(H)} \left( H(u)_1 \le Y_{u}(t) \le t+H(u)_2 \wedge Y_{u}(t) \in \mathbb{Z} \right)\right). \end{displaymath}

Then $S(H)(t)$ contains exactly one point of $V(\alpha,t, \mathbf{y_2})$ if $\mathbb{P}(\alpha,t, \mathbf{y_2}, H)$ and zero points otherwise. \end{proposition}

\begin{proof}
By definition, for all $\mathbf{y_3}$ in $S(H)(t)$ and $u$ in $U'(H),$ $y_u = I(H(u))(t).$ Proposition \ref{J2} tells us the unique real vector $\mathbf{y_3}$ with these coordinates which satisfies $A(\mathbf{y_2} \oplus \mathbf{y_3}) = \mathbf{b(\alpha)}(t),$ so if this vector lies in $S(H)(t)$ (which is a subset of $\{0, \ldots, t-1\}^{n_1}$), then $S(H)(t)$ contains exactly one point of $V(\alpha,t, \mathbf{y_2})$ and zero otherwise. We automatically have, for $u$ in $U'(H),$ that $y_u \in \mathbb{Z}.$ Therefore, $\mathbb{P}(\alpha,t, \mathbf{y_2}, H)$ is equivalent to the proposition that $\mathbf{Y_3}(\alpha, H)(\mathbf{y_2}, t)$ is in $S(H)(t).$
\end{proof}

This proposition is not a finite logical combination of parametric inequalities because the clauses $Y_u(t) \in \mathbb{Z}$ are not parametric inequalities in general. Fortunately, we can form an auxiliary PILP in which an equivalent proposition \textit{is} a finite logical combination of parametric inequalities. This auxiliary PILP will have many more indeterminates. These indeterminates act as remainders. When a remainder is $0,$ a certain expression is an integer.

For convenience, we find a common denominator. Each coordinate of the vector $\mathbf{Y_3}(\alpha, H)(\mathbf{y_2}, t)$ can be written as a $\mathbb{Q}[t]$-covector times $\mathbf{y_2}$ plus an element of $\mathbb{Q}[t]$. Therefore, the whole vector has a common denominator. Since there are finitely many $\alpha$ and $H,$ there exists a common denominator to all vectors $\mathbf{Y_3}(\alpha, H)(\mathbf{y_2}, t)$ say $T,$ which we take to be positive. 

\subsection{Step 4}
In this subection, we define a PILP $Q^*$ which is not in form shown in Theorems \ref{Chen} or \ref{Shen}. Instead, its feasible set is defined by a disjunction of conjunctions of parametric inequalities rather than a conjunction. 

We refer to the data described in Theorem \ref{main3}. (Recall that this section is the bulk of the argument to prove this theorem.)  The indeterminates are $\mathbf{y_2}=(y_1, \ldots, y_{n_2})$ and $Y_{\alpha, H, u, i}$ for all $\alpha$ in $\mathfrak{S},$ $H$ of maximal depth, $u$ in $U(H),$ and $i=1$ or $2$ (so there are finitely many indeterminates). The indeterminates are constrained to be integers, as usual. Let $\mathbf{y}$ be a vector of all of the indeterminates; the order is not important. The objective function is $\mathbf{c}^\intercal(t) \mathbf{y_2},$ which can be written as a polynomial covector times $\mathbf{y}.$ 

%%%%%%%%%%%%%%%%%% no two points in L^*(t)?
Some of the constraints of $Q^*$ are $A_2(t) \mathbf{y_2} \le \mathbf{b_2}(t)$ (which are the constraints of $L_2(t)$) and  $\\ 0 \le Y_{\alpha, H, u, 2} \le T-1.$ For all $\alpha$ in $\mathfrak{S},$ $H$ of maximal depth, and $u$ in $U(H),$ we impose the constraint
\begin{displaymath} T \mathbf{Y_3}(\alpha, H)(\mathbf{y_2}, t)_u = Y_{\alpha, H, u, 1} T + Y_{\alpha, H, u, 2}, \end{displaymath} where $T$ is the common denominator from the end of the previous subsection. These can be written as parametric inequalities. Note that by assumption, these constraints imply $0 \le y_1, \ldots, y_{n_2} \le t-1.$ 

Let its lattice point set be $L^*(t).$ It is easy to see that for all $t,$ no two points in $L^*(t)$ have the same value of $\mathbf{y_2}.$ 

Define $\mathbb{P}(\alpha, t, \mathbf{y_2})$ as in Proposition \ref{J1}; it is a finite logical combination of parametric inequalities in $\mathbf{y_2}.$ Define $\mathbb{P}'(\alpha, t, \mathbf{y_2}, H)$ (similarly to $\mathbb{P}(\alpha, t, \mathbf{y_2}, H)$ in Proposition \ref{J3}) to be
\begin{displaymath} \left(\bigwedge_{u=n_2+1}^{n_1+n_2} 0 \le Y_u \le t-1 \right) \wedge \left(\bigwedge_{i=1}^{q-q_1} \left( H(v_i)_1 \le Y_{v_i}(t) \le t+H(v_i)_2 \wedge Y_{\alpha, H, u, 2}=0 \right)\right), \end{displaymath}
where $\mathbf{Y_3}(\alpha, H)(\mathbf{y_2}, t)=(Y_{n_2+1}, \ldots, Y_{n_1+n_2}).$ Let $\mathcal{H}$ be the set of all $H$ such that $S(H)$ has maximal depth. The last constraint of $Q^*$ is \begin{align} \label{E4}
\bigvee_{W_1 \subset \mathfrak{S}} \left( \left( \bigwedge_{\alpha \in \mathfrak{S} \setminus W_1} \neg \mathbb{P}(\alpha, t, \mathbf{y_2}) \right) \wedge \left( \bigvee_{\substack{W_2 \subset W_1 \times \mathcal{H} \\ |W_1||\mathcal{H}|-|W_2| < m}} \bigwedge_{(\alpha, H) \in W_2} \neg \mathbb{P}'(\alpha, t, \mathbf{y_2}, H) \right) \right). \end{align}

Observe that this last constraint is a finite combination of parametric inequalities in $\mathbf{y}$. Let $Q^*$ have optimum value functions $\{f^*_\ell\}$ and size function $g^*.$ 

The following lemma is an quick check that we need to apply Theorems \ref{Chen} and \ref{Shen}.

\begin{lemma} Let $R^*(t)$ be the real vector set. Then for all $t,$ $R^*(t)$ is bounded. \end{lemma}
\begin{proof} Fix $t.$ Then $R_2(t)$ is bounded, so $R^*(t)$ is bounded in the coordinates $\mathbf{y_2}=(y_1, \ldots, y_{n_2}),$ say by $N$ in each coordinate. $R^*(t)$ is bounded in the coordinates $Y_{\alpha, H, u, 2}.$ For $\mathbf{y}$ in $R^*(t),$ $\\ |Y_{\alpha, H, u, 1}| \le | Y(\alpha, \mathbf{y_2}, H)_u(t)|+1.$ We have $|y_i| < N,$ and $Y(\alpha, \mathbf{y_2}, H)_u(t)$ equals a real covector times $\mathbf{y_2},$ so $R^*(t)$ is bounded in the coordinates $Y_{\alpha, H, u, 1},$ as desired.
\end{proof} 

\begin{proposition} \label{Q} 
Fix $t.$ A vector $\mathbf{y_2}$ is the projection of exactly one element of $L^*(t)$ if it lies in $L_3(t)$, and $\mathbf{y_2}$ is the projection of zero elements of $L^*(t)$ otherwise. \end{proposition}

\begin{proof}
\begin{comment}
Fix $t.$ Given an integer vector $\mathbf{y_2},$ there is a unique integer vector $\mathbf{y}$ that projects onto $\mathbf{y_2}$ such that $Y_{\alpha, H, u, i}$ such that the constraints $0 \le Y_{\alpha, H, u, 2} \le T-1$ and 
\begin{displaymath} T Y(\alpha, \mathbf{y_2}, H)_u(t) = Y_{\alpha, H, u, 1} T + Y_{\alpha, H, u, 1}\end{displaymath}
are satisfied. Therefore $\mathbf{y_2}$ is the projection of one element of $L^*(t)$ if $\mathbf{y}$ satisfies the other constraints of $Q^*,$ and it is the projection of zero elements otherwise. \end{comment}

Suppose that $\mathbf{y_2}$ is the projection of one element of $L^*(t).$ Then $\mathbf{y_2}$ lies in $L_2(t).$ because the PILP $Q^*$ contains the constraints on $\mathbf{y_2}$ that define $L_2(t)$ (along with other constraints). From (\ref{E4}), there exists a subset $W_1$ of $\mathfrak{S}$ and a subset $W_2$ of $W_1 \times \mathcal{H}$ that excludes less than $m$ elements such that

\begin{displaymath} \left( \bigwedge_{\alpha \in \mathfrak{S} \setminus W_1} \neg \mathbb{P}(\alpha, t, \mathbf{y_2})\right) \wedge \left( \bigwedge_{(\alpha, H) \in W_2} \neg \mathbb{P}'(\alpha, t, \mathbf{y_2}, H)\right). \end{displaymath}

By Proposition \ref{J1}, for all $\alpha$ in $\mathfrak{S} \setminus W_1,$ $\mathbf{y_2}$ is not in $J(\alpha, t),$ and $\# V(\alpha, t, \mathbf{y_2})=0.$  

There exists a unique integer vector $\mathbf{y}$ that projects onto $\mathbf{y_2}$ such that the constraints 
\begin{displaymath} 0 \le Y_{\alpha, H, u, 2} \le T-1 \text{   and   }T  \mathbf{Y_3}(\alpha, H)(\mathbf{y_2}, t)_u = Y_{\alpha, H, u, 1} T + Y_{\alpha, H, u, 2}\end{displaymath}
are satisfied. We have assumed that $\mathbf{y_2}$ is the projection of one element of $L^*(t),$ so $\mathbf{y}$ satisfies the other constraints of $Q^*.$

It's easy to check that the proposition $Y_{\alpha, H, u, 1}=0$, which is implicitly a proposition of $\mathbf{y_2},$ is equivalent to the proposition $Y(\alpha, \mathbf{y_2}, H)_u(t) \in \mathbb{Z}.$ Then $\mathbb{P}'(\alpha, t, \mathbf{y_2}, H)$ is equivalent to $\mathbb{P}(\alpha, t, \mathbf{y_2}, H)$. By Proposition \ref{J3}, $\neg \mathbb{P}(\alpha, t, \mathbf{y_2}, H)$ implies that $S(H)(t)$ contains zero points of $V(\alpha, t, \mathbf{y_2}).$ For all $(\alpha, H)$ in $W_1 \times \mathcal{H}$ except at most $m-1,$ $S(H)(t)$ contains zero points of $V(\alpha, t, \mathbf{y_2}).$

Therefore, for all $\alpha$ in $W_1,$ the number of distinct elements of $V(\alpha, t, \mathbf{y_2})$ that lie in a subset of maximal depth is less than $m.$ By Corollary \ref{c321}, this number equals $\# V(\alpha, t, \mathbf{y_2}).$ The sets $K(\alpha, t)$ for $\alpha \in \mathfrak{S}$ are disjoint, so 
\begin{align*} &\color{white}{=} \color{black}\, \# \{ \mathbf{y_3} \in \{0, \ldots, t-1\} \mid (\mathbf{y_2} \oplus \mathbf{y_3}) \in \sqcup_{\alpha \in \mathfrak{S}} K(\alpha, t)\}\\
& = \sum_{\alpha \in \mathfrak{S}} \# \{ \mathbf{y_3} \in \{0, \ldots, t-1\} \mid (\mathbf{y_2} \oplus \mathbf{y_3}) \in K(\alpha, t)\} \\
&= \sum_{\alpha \in \mathfrak{S}, H \in \mathcal{H}} \# V(\alpha, t, \mathbf{y_2}) \cap S(H)(t) \le m-1, \end{align*}
so $\mathbf{y_2}$ lies in $L_3(t).$

Conversely, suppose that $\mathbf{y_2}$ lies in $L_3(t).$ There is a unique integer vector $\mathbf{y}$ that projects onto $\mathbf{y_2}$ such that the constraints $0 \le Y_{\alpha, H, u, 2} \le T-1$ and 
\begin{displaymath} T \mathbf{Y_3}(\alpha, H)(\mathbf{y_2}, t)_u = Y_{\alpha, H, u, 1} T + Y_{\alpha, H, u, 2}\end{displaymath}
are satisfied. $\mathbf{y_2}$ is the projection of one element of $L^*(t)$ if $\mathbf{y}$ satisfies the other constraints of $Q^*.$ %%%%%%% changed Q to Q^*. correct?

The set $L_3(t)$ is a subset of $L_2(t),$ so $\mathbf{y}$ satisfies the constraints $A_2(t) \mathbf{y_2} \le \mathbf{b_2}(t).$ By definition, 

\begin{align*} m &> \# \{ \mathbf{y_3} \in \{0, \ldots, t-1\} \mid (\mathbf{y_2} \oplus \mathbf{y_3}) \in \sqcup_{\alpha \in \mathfrak{S}} K(\alpha, t)\}  \\
&= \sum_{\alpha \in \mathfrak{S}} \# \{ \mathbf{y_3} \in \{0, \ldots, t-1\} \mid (\mathbf{y_2} \oplus \mathbf{y_3}) \in K(\alpha, t)\}  \\
&\ge \sum_{\alpha \in \mathfrak{S}, H \in \mathcal{H}} \# V(\alpha, t, \mathbf{y_2}) \cap S(H)(t). \end{align*}

Let 
\begin{align*} W_1 &:= \{ \alpha \in \mathfrak{S} \mid \mathbb{P}(\alpha, t, \mathbf{y_2})\}, \\ W_2 &:= \{(\alpha, H) \in W_1 \times \mathcal{H} \mid V(\alpha, t, \mathbf{y_2}) \cap S(H)(t) = \varnothing\}. \end{align*}

Since the sets $V(\alpha, t, \mathbf{y_2}) \cap S(H)(t)$ have size at most 1, $|W_1||\mathcal{H}|-|W_2| < m.$ By Proposition \ref{J3}, for $(\alpha, H) \in W_2,$ $\neg \mathbb{P}(\alpha, t, \mathbf{y_2}, H)$ holds. As before, $\mathbb{P}(\alpha, t, \mathbf{y_2}, H)$ is equivalent to $\mathbb{P'}(\alpha, t, \mathbf{y_2}, H)$. Altogether, we have 
\begin{displaymath} \left( \bigwedge_{\alpha \in \mathfrak{S} \setminus W_1} \neg \mathbb{P}(\alpha, t, \mathbf{y_2})\right) \wedge \left( \bigwedge_{(\alpha, H) \in W_2} \neg \mathbb{P}'(\alpha, t, \mathbf{y_2}, H)\right), \end{displaymath}
which implies (\ref{E4}) and that $\mathbf{y}$ lies in $L^*(t).$
\end{proof}

\begin{corollary} For all $t$ and $\ell,$ $g(t)=g^*(t)$ and $f_\ell(t)=f^*_\ell(t).$ \end{corollary}

\begin{proof}
By Proposition \ref{Q}, the map $\varphi_t$ that takes a vector $\mathbf{y}$ to $\mathbf{y_2}$ by ignoring the other coordinates is a bijection from $L^*(t)$ to $L_3(t).$ By construction, this bijection commutes with evaluating the respective objective function, so we're done.
\end{proof}

\section{Untangling Disjunctions and Conjunctions}
\label{sec-untangling}

Theorems \ref{Chen} and \ref{Shen} only apply to PILPs given by a finite conjunctions of parametric inequalities; they do not trivially apply to $Q^*.$ In this section, we extend these theorems to PILPs such as $Q^*$. Recall that the full constraints of $Q^*$ are 
\begin{displaymath} A_2(t) \mathbf{y_2} \le \mathbf{b_2}(t) \wedge \left( \bigwedge 0 \le Y_{\alpha, H, u, 2} \le T-1 \right) \wedge \left( \bigwedge T  \mathbf{Y_3}(\alpha, H)(\mathbf{y_2}, t)_u = Y_{\alpha, H, u, 1} T + Y_{\alpha, H, u, 2} \right) \wedge (\ref{E4}). \end{displaymath} 

The negation of a parametric inequality over integer indeterminates is another parametric inequality. Using this fact, De Morgan's laws and the logical distributive laws, one finds that the constraints of $Q^*$ are a finite disjunction of finite conjunctions of parametric inequalities of $\mathbf{y}.$

It remains to extend theorems \ref{Chen} and \ref{Shen} to PILPs given by a finite disjunction of finite conjunctions of parametric inequalities whose real vector set is bounded for all $t.$

\begin{proposition} For a PILP whose lattice point set is given by a finite disjunction of finite conjunctions of parametric inequalities whose real vector set is bounded for all $t,$ all optimum value functions and the size value function are EQP. \end{proposition}

\begin{proof} First, we rewrite a finite disjunction of finite conjunctions of parametric inequalities as a finite ``disjoint disjunction" of finite conjunctions of parametric inequalities; that is, the component conjunctional clauses define parametric regions which are disjoint for any fixed $t.$ We generalize the following identity: 
\begin{displaymath} (A \wedge B) \vee (C \wedge D) = (A \wedge B) \oplus (A \wedge \neg B \wedge C \wedge D) \oplus (\neg A \wedge C \wedge D). \end{displaymath}

Here, $\oplus$ denotes disjunction and simultaneously states that no two of the following propositions can be true at the same time. This is \textit{not} exclusive disjunction; it is analogous to a disjoint union of sets. Let $T$ be a finite set of finite sets of propositions and suppose that all sets are totally ordered by $<.$ 

\begin{align*} \bigvee_{S \in T}  \bigwedge_{a \in S} a  &= \bigoplus_{S \in T}  \left( \left(\bigwedge_{a \in S} a \right) \wedge \bigwedge_{R \in T, R < S} \neg \left( \bigwedge_{a \in R} a \right)\right)  \\
&= \bigoplus_{S \in T}  \left( \left(\bigwedge_{a \in S} a \right) \wedge \bigwedge_{R \in T, R < S}\bigoplus_{a \in R} \left( a \wedge \bigwedge_{b \in R, b < a} \neg b \right) \right) \\
 &= \bigoplus_{S \in T}  \left( \left(\bigwedge_{a \in S} a \right) \wedge \bigoplus_{(a_R \in R) \forall R \in T, R < S} \bigwedge_{R \in T, R < S}\left(a_R \wedge \bigwedge_{b \in R, b < a_R} \neg b \right) \right)\\ 
&= \bigoplus_{\substack{S \in T\\(a_R \in R) \forall R \in T, R < S}}  \left( \left(\bigwedge_{a \in S} a \right) \wedge \bigwedge_{R \in T, R < S}\left(a_R \wedge \bigwedge_{b \in R, b < a_R} \neg b \right) \right). \end{align*}

Observe that the last line is a finite disjoint disjunction of finite conjunctions of propositions. Now consider a PILP whose constraints are equivalent to a finite disjunction of finite conjunctions of parametric inequalities. Let $T$ be the set that represents these constraints. By the above equations, the lattice point set (or real vector set) equals the disjoint union of the lattice point sets (or real vector sets) for the PILPs determined by the component conjunctions. If the real vector set defined by $\vee_{S \in T}  \wedge_{a \in S} a$ is bounded for all $t$, the same must be true for the component conjunctional clauses in the last line. Therefore, Theorems \ref{Chen} and \ref{Shen} apply directly to each component. 

Now, we extend Theorems \ref{Chen} and \ref{Shen} to finite disjoint disjunctions. It is easy to see that Theorem \ref{Chen} extends to finite disjoint disjunctions because the size functions add, and a finite sum of EQPs is EQP. 

To extend Theorem \ref{Shen}, observe that for all $t$ and $\ell,$ the $\ell^{\text{th}}$ largest value of the objective function in the disjoint union equals the $\ell^{\text{th}}$ largest value in the union of the multisets of the $\ell$ largest values from each component (at $t$). If a component's lattice point set has size less than $\ell,$ our convention is that the multiset of the $\ell$ largest values contains $-\infty$ at least once. There are finitely many components. By Proposition 3.4 of \cite{Shen}, which is copied below, the $\ell^{\text{th}}$ optimum value function is EQP.

\begin{proposition}[{Proposition 3.4 in \cite{Shen}}] Let $m$ and $\ell$ be positive integers and $f_{1}, \ldots,  f_m$ be eventual quasi-polynomials with codomain $\{-\infty\} \cup \mathbb{Z}$. For all $t,$ let $f(t)$ be the $\ell^{\text{th}}$ largest value among the multiset $\{f_1(t), \ldots, f_m(t)\}.$ Then $f$ is eventually quasi-polynomial.
\end{proposition}

\end{proof}

This proposition applies directly to $Q^*$ and completes the proof of Theorem \ref{main3}.

\section{Proof of Theorem \ref{main}}
\label{sec:main}
Roune and Woods showed that, to prove their conjecture, it sufficed to consider certain $n$-tuples of polynomials. We present the same results for our more general theorem.

\begin{proposition} \label{red1} To prove Theorem \ref{main}, it suffices to consider polynomials $Q_1, \ldots, Q_n$ that map $\mathbb{Z}$ to $\mathbb{Z}$ and have positive leading coefficients such that for $t \gg 0,$ $\gcd(Q_1(t), \ldots, Q_n(t))=1.$ \end{proposition}

\begin{proof} Let $P_1, \ldots, P_n$ be as in Theorem 1.5. Lemma 3.4 of \cite{RW} by Roune and Woods shows that $h(t):=\gcd(P_1(t), \ldots, P_n(t))$ is EQP. Each $P_i$ is eventually positive, so in fact $h$ is eventually positive. By Lemma 3.3 of their paper, for $i=1, \ldots, n,$ $\frac{P_i(t)}{h(t)}$ is EQP. It is also integer valued and positive. Suppose that $d$ is a common period of $h$ and all polynomials $P_i/h$, and let $a$ be an integer. For $i=1, \ldots, n,$ let $P_{i, a}(s):=\frac{P_i(a+sd)}{h(a+sd)}.$ 

For all $i,$ $d$ is a period of $P_i/h,$ so $P_{i, a}$ eventually agrees with a polynomial. Furthermore, $P_{i, a}$ is eventually positive and maps $\mathbb{Z}$ to $\mathbb{Z},$ and for $s \gg 0,$ $\gcd(P_{1, a}(s), \ldots, P_{n, a}(s))=1.$

Suppose that Theorem \ref{main} is true for polynomials $Q_1, \ldots, Q_n$ such that for sufficiently large $t,$ $\gcd(Q_1(t),  \ldots, Q_n(t))=1.$ Since we only care about large $s,$ this theorem applies to $P_{1, a}, \ldots, P_{n, a}$, so $F_{m,\ell}(P_{1,a}(s), \ldots, P_{n, a}(s))$ as a function of $s$ is EQP. When $t \equiv a \pmod{d},$ $\frac{t-a}{d}$ is an integer, and

\begin{align*} F_{m,\ell}(P_1(t), \ldots, P_n(t))&=h(t)F_{m, \ell}\left(\frac{P_1(t)}{h(t)},  \ldots, \frac{P_n(t)}{h(t)}\right)\\
&= h(t) F_{m, \ell} \left( P_{1, a}\left(\frac{t-a}{d}\right), \ldots, P_{n, a} \left( \frac{t-a}{d} \right) \right). \end{align*}

Since $h$ and $F_{m, \ell}(P_{1, a}(s), \ldots, P_{n, a}(s))$ are EQP, $F_{m, \ell}(P_1(t), \ldots, P_n(t))$ with $t$ restricted to $a \pmod{d}$ is EQP. It easily follows that $F_{m, \ell}(P_1(t), \ldots, P_n(t))$ is EQP. A similar argument proves the proposition for $G_m.$

\end{proof}
\begin{proposition} \label{red2} To prove Theorem \ref{main}, it suffices to consider $P_1, \ldots, P_n$ in $\mathbb{Z}[u]$ such that for $t \gg 0,$ $\gcd(P_1(t), \ldots, P_n(t))=1.$
\end{proposition}

\begin{proof}Given Proposition \ref{red1}, this proposition is merely Lemma 3.1 and Remark 3.2 of Roune and Woods \cite{RW}.\end{proof}

We require a simple bound on the Frobenius number, which corresponds to the second region of a parametric exclusion problem.

\begin{theorem}[Erd\H{o}s and Graham\cite{EG}] \label{EG}
Let $x_1, \ldots, x_n$ be positive integers such that $x_1 < x_2 \cdots < x_n$ and $\gcd(x_1, \ldots, x_n)=1.$ Then $F(x_1, \ldots, x_n) \le 2 x_n \left \lfloor \frac{x_1}{n} \right \rfloor - x_1.$
\end{theorem}
It is simpler to prove the weaker result $F(x_1, \ldots, x_n) < x_n^2.$ The bound that we really need is the following.

\begin{corollary}
Let $n, m, $ and $\ell$ be positive integers with $n \ge 2$ and $P_1, \ldots, P_n$ be in $\mathbb{Z}[u]$ and eventually positive. Then for some integer $r$ all $t\gg 0,$ $\ell+F_{m, \ell}(P_1(t), \ldots, P_n(t)) < t^r.$
\end{corollary}

\begin{proof} This follows easily from Theorem \ref{EG} and the proof of Proposition \ref{p1}, which shows that $h(k)>m$ for $k>m a_1 a_2 + F(a_1, \ldots, a_n).$ \end{proof}

\begin{proof}[Proof of Theorem \ref{main}] By Propositions \ref{red1} and \ref{red2}, it suffices to consider the following type of parametric Frobenius problem. Let $n, m,$ and $\ell$ be positive integers such that $n \ge 2.$ Let $P_1, \ldots, P_n$ be polynomials in $\mathbb{Z}[u]$ with positive leading coefficient such that for sufficiently large $t,$ $\gcd(P_1(t), \ldots, \\ P_n(t))=1.$ Then there exists $r$ as in the above corollary. 

We now formulate a parametric exclusion problem as in theorem \ref{main2}. Let $n_1=n, n_2=1, m=m,$ and the indeterminates be $k, b_1, \ldots, b_n.$ Let 
\begin{displaymath} R_1(t)=\left\{(k, b_1, \ldots, b _n) \in [0, t^r-1]^{n+1} \mid k - \sum_{i=1}^ni b_i P_i(t) = \ell\right\}, \quad R_2(t) = [0, t^r-1]. \end{displaymath}

It is easy to see how to form $A_1, A_2, \mathbf{b_1}, \mathbf{b_2}$  such that $R_1$ and $R_2$ match the hypotheses in Theorem \ref{main}. We are specifying $\ell$ on the right hand side because of a rather minor technicality: when there are less than $\ell$ nonnegative integers which are the sum of $P_1(t), \ldots, P_n(t)$ in less than $m$ ways, then we have defined $F_{m, \ell}(P_1(t), \ldots, P_n(t))$ to be a negative integer which is necessarily at least $-\ell.$ Therefore, it is helpful to translate the parametric Frobenius problem by $+\ell$ so that $F_{m, \ell}$ is guaranteed to be nonnegative.

The sets $R_1(t)$ and $R_2(t)$ are bounded for all $t.$ Let $L_1(t), L_2(t),$ and $L_3(t)$ be as in the hypotheses. Let $\mathbf{c} = (1).$ Let $\{f_\ell\}$ be the optimum value functions and $g$ be the size function. By Theorem \ref{main2}, $f_\ell$ and $g$ are EQP for all $\ell.$

For $t \gg 0,$ $\ell+F_{m, \ell}(P_1(t), \ldots, P_n(t)) < t^r$ and $P_1(t), \ldots, P_n(t)>1.$ Since $\gcd(P_1(t), \ldots, P_n(t)) = 1,$ we also have $\ell+F_{m, \ell}(P_1(t), \ldots, P_n(t)) \ge 0.$ For $t \gg 0,$ $L_3(t)$ is the set of all nonnegative integers $k$ such that $k-\ell$ is a nonnegative integer combination of $P_1(t), \ldots, P_n(t)$ in less than $m$ ways, since such $k$ must be less than $t^r,$ and a coefficient in any nonnegative integer combination cannot exceed $k.$ Therefore,
\begin{displaymath} F_{m, \ell}(P_1(t), \ldots, P_n(t)) = -\ell + f_{\ell}(t).\end{displaymath}

Since the EQP property only depends on large arguments, $ F_{m, \ell}(P_1(t), \ldots, P_n(t))$ is an EQP function of $t.$ For $t \gg 0,$ $g(t)$ counts all integers at least $-\ell$ which are nonnegative integer combinations in less than $m$ ways, and all of the integers $-\ell, \ldots, -1$ are counted. However, $G_m(P_1(t), \ldots, P_n(t))$ counts only the positive integers. Therefore, 
\begin{displaymath} G_m(P_1(t), \ldots, P_n(t)) = -\ell + g(t), \end{displaymath}

from which it follows that $G_m(P_1(t), \ldots, P_n(t))$ is an EQP function of $t,$ as before.\end{proof}

\begin{remark}
Theorems \ref{main} and \ref{main2} are not true if constants $m$ and $\ell$ are replaced by polynomials of $t$ such as $t$ itself. For example, $F_{1, t}(t, t-1),$ $F_{t, 1}(6, 10, 15),$ and $G_t(6, 10, 15)$ are not eventually quasi-polynomial. We omit the proofs of these statements. \end{remark}

\section{Future directions}

Our work does not address how the components of the resulting eventual quasi-polynomials are related. For example, the components of $F_{m, \ell}(P_1(t), \ldots, P_n(t))$ are likely to have the same degree fairly often, but from our argument, it is very hard to tell when. Our methods do not offer reasonable bounds on the period of the eventual quasi-polynomials or an algorithmically feasible way of computing them. For these reasons, it may be of interest to refine our argument.

A parametric problem closely related to the parametric Frobenius problem is the identification of a reduced Gr\"{o}bner basis of an associated ideal. See \cite{Rou}. The reduced Gr\"obner basis is not unique, so we define a particular one here and also rewrite the ideas without referring to commutative algebra. For positive integers $a_1, \ldots, a_n,$ let
\begin{displaymath} S:=\{\mathbf{v} \in \mathbb{Z}^n \mid \mathbf{v} \cdot (a_1, \ldots, a_n) = 0 \}. \end{displaymath}

For real $n$-dimensional vectors, write $\prec$ for lexicographically less than, etc. and $<$ for less than in all coordinates, etc. For $\mathbf{v}$ (in $\mathbb{R}^n,$) such that $\mathbf{v} \prec 0,$ we define the positive part of $\mathbf{v},$ $\mathbf{v}^+,$ such that $\mathbf{v}^+_i=\max(v_i, 0),$ and we define $\mathbf{v}^-$ to be $\mathbf{v}^+-\mathbf{v}.$ Let 
\begin{displaymath} G(S):=\{\mathbf{v} \in S \mid \mathbf{v} \prec 0 \wedge \forall \mathbf{u} \in S, (\mathbf{u} \prec 0 \wedge \mathbf{u}^+ \le \mathbf{v}^+) \implies (\mathbf{u}^+ = \mathbf{v}^+ \wedge \mathbf{v}^- \preceq \mathbf{u}^-) \}. \end{displaymath}

In other words, $G(S)$ is the set of the elements of $S$ which are lexicographically less than $0$ whose positive parts do not strictly dominate ($\ge$ in all coordinates and $>$ in at least one) the positive part of any other such element of $S$ and whose negative part is the lexicographically minimum possible for that positive part. It is not hard to show that $G(S)$ is finite for each finite set of positive integers. As before, we let $P_1, \ldots, P_n$ be in $\mathbb{R}[u]$ which take $\mathbb{Z}$ to $\mathbb{Z}$ and are eventually positive. For sufficiently large $t,$ let $S(t)$ be 
\begin{displaymath} \{\mathbf{v} \in \mathbb{Z}^n \mid \mathbf{v} \cdot (P_1(t), \ldots, P_n(t)) = 0 \}, \end{displaymath}
and let $G(S(t))$ be as above. 

\begin{conjecture} There exists a PILP given by a finite disjunction of finite conjunction of parametric inequalities whose lattice point set equals $G(S(t))$ for $t \gg 0.$ The same is true for $\{ \mathbf{v}^+ \mid \mathbf{v} \in G(S(t)) \}.$ \end{conjecture}

\begin{remark} G(S)(t) may not be of bounded size. One can show that when $n=4,$ $P_1(t)=\\ 2t^2+4t, P_2(t)=2t^2-t+1, P_3(t)=4t^2+2t-1, P_4(t)=4t^2-3t,$ for $t \gg 0,$ 
\begin{displaymath} G(S(t)) = \{(-1,1, 1, -1), (-2t+2, -1, t-1, 2), \ldots \langle +1, -1, -1, +1 \rangle \ldots, (-t+1, -t, 0, t+1),  \end{displaymath} 
\begin{displaymath}  (0, -2t, t, -1), (-2t+1, 2t, 0, 2) \ldots \langle +1, -1, -1, +1 \rangle \ldots (-t-1, t+2, -t+2, t), (-1, 2t+1, -t+1, 0) \}. \end{displaymath} \end{remark}

The original motivation for formulating this conjecture was propositions 4 and 5 of Roune's \cite{Rou}, which showed that the Frobenius number, as a function of $t,$ is the maximum of a fixed polynomial covector times $\mathbf{v}$ for $\mathbf{v}$ in some set related to $G(S(t)).$ If $G(S(t))$ can be understood, then Theorem \ref{Shen} can be applied. We ended up using a more direct approach to decompose the parametric Frobenius problem.

Our work shows that the parametric Frobenius problem can be well understood by transforming it into a parametric integer linear program. This idea may be useful in studying other parametric combinatorics problems such as those in \cite{Woo} whose answers are suspected to be eventually polynomial.

\section{Acknowledgements}

The author thanks Joe Gallian for suggesting the problem and for helpful comments on the manuscript. This research was conducted at the University of Minnesota Duluth REU and was supported by NSF grant 1358695 and NSA grant H98230-13-1-0273.

\end{document}